\documentclass[11pt]{article} 

\usepackage{amssymb}
\usepackage{amsmath,amsthm,amsfonts,epsfig,setspace}

\usepackage[utf8]{inputenc}
\usepackage{cases}
\usepackage{todonotes}

\usepackage{graphicx} 
\usepackage{subcaption}
\usepackage{color}
\usepackage{hyperref}
\usepackage{tikz}
\usepackage{wrapfig}

\usepackage{soul}
\sethlcolor{cyan}
\usepackage{xcolor}
\newcommand{\mathhl}[1]{\colorbox{cyan}{$\ds #1$}}

\renewcommand{\hl}[1]{#1} 
\renewcommand{\mathhl}[1]{#1}

\usetikzlibrary{intersections,decorations.pathreplacing,decorations.markings,calc,positioning}
\graphicspath{{images/}}
\numberwithin{equation}{section}

\usepackage{geometry} 
\newcommand{\nm}{\noalign{\smallskip}}
\newcommand{\beq}{\begin{equation}}
\newcommand{\eeq}{\end{equation}}
\newcommand{\eqnref}[1]{(\ref {#1})}
\newcommand{\ds}{\displaystyle}
\newcommand{\pf}{\noindent {\sl Proof}. \ }
\newcommand{\p}{\partial}
\newcommand{\pd}[2]{\frac {\p #1}{\p #2}}

\newcommand{\ie}{\textit{i.e.}}

\newtheorem{theorem}{Theorem}[section]

\newtheorem{lemma}[theorem]{Lemma}
\newtheorem{prop}[theorem]{Proposition}

\newtheorem{remark}{Remark}
\newtheorem{asump}{Assumption}[section]

\newcommand{\R}{\mathbb{R}}

\newcommand{\Bx}{{x}}
\newcommand{\By}{{y}}

\newcommand{\C}{\mathcal{C}}

\newcommand{\Scal}{\mathcal{S}}
\newcommand{\Kcal}{\mathcal{K}}
\newcommand{\Acal}{\mathcal{A}}

\newcommand{\B}{\mathcal{B}}
\newcommand{\iu}{\mathrm{i}\mkern1mu}
\newcommand{\dx}{\: \mathrm{d}}

\title{Honeycomb-lattice Minnaert bubbles\thanks{\footnotesize The work of Hyundae Lee was supported by National Research Fund of Korea (NRF-2015R1D1A1A01059357, NRF-2018R1D1A1B07042678).}}

\author{
Habib Ammari\thanks{\footnotesize Department of Mathematics, 
	ETH Z\"urich, 
	R\"amistrasse 101, CH-8092 Z\"urich, Switzerland (habib.ammari@math.ethz.ch, brian.fitzpatrick@sam.math.ethz.ch, erik.orvehed.hiltunen@sam.math.ethz.ch, sanghyeon.yu@sam.math.ethz.ch).} \and Brian Fitzpatrick\footnotemark[2] \and Erik Orvehed Hiltunen\footnotemark[2] \and  Hyundae Lee\thanks{\footnotesize  Department of Mathematics, Inha University,  253 Yonghyun-dong Nam-gu,  Incheon 402-751,  Korea (hdlee@inha.ac.kr).}   \and Sanghyeon Yu\footnotemark[2]}

\date{} 
  
\begin{document}
\maketitle
\begin{abstract}
The ability to manipulate the propagation of waves on subwavelength scales is important for many different physical applications. In this paper, we consider a honeycomb lattice of subwavelength resonators and prove, for the first time,  the existence of a Dirac dispersion cone at subwavelength scales. As shown in [H. Ammari et al., A high-frequency homogenization approach near the Dirac points in bubbly honeycomb crystals, arXiv:1812.06178], near the Dirac points, honeycomb crystals of subwavelength resonators has a great potential to be used as near-zero materials.  Here, we perform the analysis for the example of bubbly crystals, which is a classic example of subwavelength resonance, where the resonant frequency of a single bubble is known as the {Minnaert resonant frequency}. Our first result is an asymptotic formula for the quasi-periodic Minnaert resonant frequencies. We then prove the linear dispersion relation of a Dirac cone. Our findings in this paper are numerically illustrated in the case of circular bubbles, where the multipole expansion method provides an efficient technique for computing the band structure.

\end{abstract}

\def\keywords2{\vspace{.5em}{\textbf{  Mathematics Subject Classification
(MSC2000).}~\,\relax}}
\def\endkeywords2{\par}
\keywords2{35R30, 35C20.}

\def\keywords{\vspace{.5em}{\textbf{ Keywords.}~\,\relax}}
\def\endkeywords{\par}
\keywords{Honeycomb lattice, Dirac cone, bubble, Minnaert resonance, subwavelength bandgap.}

\section{Introduction}

Subwavelength crystals are based on locally resonant structures with large material contrasts, repeated periodically. They are a novel group of synthetic materials which enables manipulation of waves on very small spatial scales, known as \emph{subwavelength} scales.  In this setting, a \emph{crystal} refers to a large structure with a periodically repeated microstructure. Due to their small scales, subwavelength crystals are very useful in physical applications, especially situations where the operating wavelengths are very large. They have been investigated both numerically and experimentally in \cite{rev1,leroy2, fink,  leroy1, leroy3, rev2, rev3}.

Recently, there have been many discoveries involving materials that exhibit intriguing wave propagation properties due to the presence of a Dirac cone in their band structures \cite{jmp, Dirac4, Dirac5,Dirac7,Dirac8,Dirac3, torrent, Dirac1, Dirac2}. A Dirac cone is a linear intersection of two curves in the dispersion diagram, and is a consequence of non-trivial symmetry of the lattice. Dirac cones have typically been studied in the context of electron bands in graphene, where peculiar effects such as Klein tunnelling and Zitterbewegung have been observed. Moreover, Dirac cones have been demonstrated in acoustic analogues of graphene, which can give effective zero refractive index materials \cite{torrent, NZ2012, dai2017dirac, zheng2020acoustic}. Typically, by breaking the symmetry of the lattice, the Dirac cone can be opened to a bandgap. This is a fundamental mechanism to create topologically robust guiding of waves \cite{Dirac4, Dirac5, Dirac6, Dirac2018, Dirac2, yves2017crytalline}.

\hl{In this work, the acoustic resonant structure is an inclusion embedded in a surrounding material with significantly higher density. Inspired by an air bubble in water, which possesses a resonant frequency known as the Minnaert resonant frequency} \cite{prsa,H3a,Minnaert1933}, \hl{we think of the system as a periodic array of bubbles. Physically, such systems can be stabilized (preserving the resonant behaviour) by replacing the liquid with a soft, elastic medium} \cite{leroy1,leroy3}. The high density contrast is crucial for the resonance to occur at subwavelength scales. Due to its simplicity, the Minnaert bubble is an ideal component for constructing subwavelength scale metamaterials, and has been studied, for instance, in \cite{prsa, H3a, AFLYZ, Ammari_David,defect,homogenization}. In particular, it was proved in \cite{AFLYZ} that a bubbly crystal features a bandgap opening in the subwavelength regime. Similarly, it is possible to create materials with Dirac cones at subwavelength scales, resulting in small scale metamaterials with Dirac singularities. Such materials have been experimentally and numerically studied in \cite{Dirac1,Dirac2,yves2018measuring}. General overviews of acoustic metamaterials are presented in \cite{alu,rev2,wu2018perspective}. 

The goal of this work is to study an acoustic analogue of graphene, composed of bubbles in a hexagonal honeycomb structure. Wave propagation in the crystal is modelled by a high-contrast Helmholtz problem. The main objective is to rigorously prove the existence of a Dirac cone in the subwavelength scale in a bubbly honeycomb crystal. 

The mathematical analysis of the band structure of graphene was originally based on a tight-binding model under certain nearest-neighbour approximations \cite{tightbind2,tightbind1}, and later generalized to a broad class of Schrödinger operators with honeycomb lattice potentials \cite{Dirac6, Dirac7, Dirac2018, Dirac3}. In this work, using layer potential theory and Gohberg-Sigal theory, we  demonstrate an original and powerful method for analysing metamaterials with honeycomb structures. 
The method is well-suited for investigating wave propagation in  media with discontinuous material parameters. Such materials arise naturally when designing subwavelength metamaterials, but are mathematically challenging to analyse. We consider a general shape of the scatterer, only assuming some natural symmetry assumptions. This generalizes previous works done with circular scatterers \cite{torrent}. Moreover, we derive an original formula for the slope of the Dirac cone in a bubbly honeycomb crystal and give explicitly the behaviour of the error term in terms of the contrast in the material parameters.

The paper is organised as follows. In Section \ref{sec:setup}, we define the geometry of the bubbly honeycomb crystal and formulate the spectral resonance problem. We also introduce some well-known results regarding the quasi-periodic Green's function on the honeycomb lattice. The computation of the Dirac cone is performed in Section \ref{sec:pre} and Section \ref{sec:dirac}. In Section \ref{sec:pre}, we derive an asymptotic formula for the quasi-periodic Minnaert resonant frequency in terms of the density contrast. In Section \ref{sec:dirac}, we rigorously show the existence of a Dirac dispersion cone. \hl{The two sections complement each other in the following way: in Section} \ref{sec:pre} \hl{we compute explicit approximations of the band functions, but cannot prove the existence of an exact Dirac cone. Conversely, in Section} \ref{sec:dirac} \hl{we  prove the existence of a Dirac cone, but we cannot explicitly compute the slope and centre of the cone with the method used there.} Also, it is worth emphasizing that the high-contrast condition is needed to guarantee the two-fold degeneracy found in a Dirac cone, which could fail without this condition. Thus, having high-contrast parameters is not required due to a limitation of the proof, but should be viewed as a method to create subwavelength Dirac cones. In Section \ref{sec:num}, we numerically compute the Dirac cones in the case of circular bubbles using the multipole expansion method. The paper ends with some concluding remarks in Section \ref{sec:conclusion}.

%%%%%%%%%%%%%%%%%%%%%%%%%%%%%%%%%%%%%%%%%%%%%%%%%%%%%%%%%%%%%%%%%%%%%%%%%%%%%%%%%%%%%
\section{Problem statement and preliminaries} \label{sec:setup}
In this section, we formulate the resonance problem for the honeycomb crystal and briefly describe the layer potential theory that will be used in the subsequent analysis. 

\subsection{Problem formulation}
In order to \hl{formulate} the problem under consideration, we start by describing the bubbly honeycomb crystal depicted in Figure \ref{fig:honeycomb}. We consider a two-dimensional, infinite crystal in two dimensions. \hl{In possible physical realisations, this corresponds to a structure that is invariant along the third spatial dimension.} Define a hexagonal lattice $\Lambda$ as the lattice generated by the lattice vectors:
$$ l_1 = a\left( \frac{\sqrt{3}}{2}, \frac{1}{2} \right),~~l_2 = a\left( \frac{\sqrt{3}}{2}, -\frac{1}{2}\right).$$
Here, $a$ denotes the lattice constant. Denote by $Y$ a fundamental domain of the given lattice. Here, we take 
$$ Y:= \left\{ s l_1+ t l_2 ~|~ 0 \le s,t \le 1 \right\}. $$
We decompose the fundamental domain $Y$ into two parts:
$$ Y= Y_1\cup Y_2 , $$
where
$$ Y_1 = \left\{ s l_1+ t l_2 ~|~ 0 \le s,t, \ \text{and} \  t + s \leq 1 \right\},~ Y_2=Y\setminus Y_1.$$
We denote the centres of $Y$, $Y_1$, and $Y_2$, respectively, by $x_0, x_1$, and $x_2$, \ie{},
$$x_0 = \frac{l_1 + l_2}{2}, \quad x_1 = \frac{l_1+l_2}{3}, \quad x_2 = \frac{2(l_1 + l_2)}{3} .$$

\begin{figure}[tb]
	\centering
	\begin{tikzpicture}
	\begin{scope}[xshift=-5cm,scale=1]
	\coordinate (a) at (1,{1/sqrt(3)});		
	\coordinate (b) at (1,{-1/sqrt(3)});	
	\pgfmathsetmacro{\rb}{0.25pt}
	\pgfmathsetmacro{\rs}{0.2pt}
	
	\draw (0,0) -- (1,{1/sqrt(3)}) -- (2,0) -- (1,{-1/sqrt(3)}) -- cycle; 
	\begin{scope}[xshift = 1.33333cm]
	\draw plot [smooth cycle] coordinates {(0:\rb) (60:\rs) (120:\rb) (180:\rs) (240:\rb) (300:\rs) };
	\end{scope}
	\begin{scope}[xshift = 0.666667cm, rotate=60]
	\draw plot [smooth cycle] coordinates {(0:\rb) (60:\rs) (120:\rb) (180:\rs) (240:\rb) (300:\rs) };
	\end{scope}
	
	\draw[opacity=0.2] ({2/3},0) -- ({4/3},0)
		($0.5*(1,{1/sqrt(3)})$) -- ({2/3},0)
		($0.5*(1,{-1/sqrt(3)})$) -- ({2/3},0)
		($(1,{1/sqrt(3)})+0.5*(1,{-1/sqrt(3)})$) -- ({4/3},0)
		($0.5*(1,{1/sqrt(3)})+(1,{-1/sqrt(3)})$) -- ({4/3},0);

\begin{scope}[shift = (a)]
	\begin{scope}[xshift = 1.33333cm]
	\draw plot [smooth cycle] coordinates {(0:\rb) (60:\rs) (120:\rb) (180:\rs) (240:\rb) (300:\rs) };
	\end{scope}
	\begin{scope}[xshift = 0.666667cm, rotate=60]
	\draw plot [smooth cycle] coordinates {(0:\rb) (60:\rs) (120:\rb) (180:\rs) (240:\rb) (300:\rs) };
	\end{scope}	
	\draw[opacity=0.2] ({2/3},0) -- ({4/3},0)
		($0.5*(1,{1/sqrt(3)})$) -- ({2/3},0)
		($0.5*(1,{-1/sqrt(3)})$) -- ({2/3},0)
		($(1,{1/sqrt(3)})+0.5*(1,{-1/sqrt(3)})$) -- ({4/3},0)
		($0.5*(1,{1/sqrt(3)})+(1,{-1/sqrt(3)})$) -- ({4/3},0);
\end{scope}
\begin{scope}[shift = (b)]
	\begin{scope}[xshift = 1.33333cm]
	\draw plot [smooth cycle] coordinates {(0:\rb) (60:\rs) (120:\rb) (180:\rs) (240:\rb) (300:\rs) };
	\end{scope}
	\begin{scope}[xshift = 0.666667cm, rotate=60]
	\draw plot [smooth cycle] coordinates {(0:\rb) (60:\rs) (120:\rb) (180:\rs) (240:\rb) (300:\rs) };
	\end{scope}
	\draw[opacity=0.2] ({2/3},0) -- ({4/3},0)
		($0.5*(1,{1/sqrt(3)})$) -- ({2/3},0)
		($0.5*(1,{-1/sqrt(3)})$) -- ({2/3},0)
		($(1,{1/sqrt(3)})+0.5*(1,{-1/sqrt(3)})$) -- ({4/3},0)
		($0.5*(1,{1/sqrt(3)})+(1,{-1/sqrt(3)})$) -- ({4/3},0);
\end{scope}
\begin{scope}[shift = ($-1*(a)$)]
	\begin{scope}[xshift = 1.33333cm]
	\draw plot [smooth cycle] coordinates {(0:\rb) (60:\rs) (120:\rb) (180:\rs) (240:\rb) (300:\rs) };
	\end{scope}
	\begin{scope}[xshift = 0.666667cm, rotate=60]
	\draw plot [smooth cycle] coordinates {(0:\rb) (60:\rs) (120:\rb) (180:\rs) (240:\rb) (300:\rs) };
	\end{scope}
	\draw[opacity=0.2] ({2/3},0) -- ({4/3},0)
	($0.5*(1,{1/sqrt(3)})$) -- ({2/3},0)
	($0.5*(1,{-1/sqrt(3)})$) -- ({2/3},0)
	($(1,{1/sqrt(3)})+0.5*(1,{-1/sqrt(3)})$) -- ({4/3},0)
	($0.5*(1,{1/sqrt(3)})+(1,{-1/sqrt(3)})$) -- ({4/3},0);
	\end{scope}
	\begin{scope}[shift = ($-1*(b)$)]
	\begin{scope}[xshift = 1.33333cm]
	\draw plot [smooth cycle] coordinates {(0:\rb) (60:\rs) (120:\rb) (180:\rs) (240:\rb) (300:\rs) };
	\end{scope}
	\begin{scope}[xshift = 0.666667cm, rotate=60]
	\draw plot [smooth cycle] coordinates {(0:\rb) (60:\rs) (120:\rb) (180:\rs) (240:\rb) (300:\rs) };
	\end{scope}
	\draw[opacity=0.2] ({2/3},0) -- ({4/3},0)
	($0.5*(1,{1/sqrt(3)})$) -- ({2/3},0)
	($0.5*(1,{-1/sqrt(3)})$) -- ({2/3},0)
	($(1,{1/sqrt(3)})+0.5*(1,{-1/sqrt(3)})$) -- ({4/3},0)
	($0.5*(1,{1/sqrt(3)})+(1,{-1/sqrt(3)})$) -- ({4/3},0);
	\end{scope}
\begin{scope}[shift = ($(a)+(b)$)]
	\begin{scope}[xshift = 1.33333cm]
	\draw plot [smooth cycle] coordinates {(0:\rb) (60:\rs) (120:\rb) (180:\rs) (240:\rb) (300:\rs) };
	\end{scope}
	\begin{scope}[xshift = 0.666667cm, rotate=60]
	\draw plot [smooth cycle] coordinates {(0:\rb) (60:\rs) (120:\rb) (180:\rs) (240:\rb) (300:\rs) };
	\end{scope}
	\draw[opacity=0.2] ({2/3},0) -- ({4/3},0)
	($0.5*(1,{1/sqrt(3)})$) -- ({2/3},0)
	($0.5*(1,{-1/sqrt(3)})$) -- ({2/3},0)
	($(1,{1/sqrt(3)})+0.5*(1,{-1/sqrt(3)})$) -- ({4/3},0)
	($0.5*(1,{1/sqrt(3)})+(1,{-1/sqrt(3)})$) -- ({4/3},0);
	\end{scope}
	\begin{scope}[shift = ($-1*(a)-(b)$)]
	\begin{scope}[xshift = 1.33333cm]
	\draw plot [smooth cycle] coordinates {(0:\rb) (60:\rs) (120:\rb) (180:\rs) (240:\rb) (300:\rs) };
	\end{scope}
	\begin{scope}[xshift = 0.666667cm, rotate=60]
	\draw plot [smooth cycle] coordinates {(0:\rb) (60:\rs) (120:\rb) (180:\rs) (240:\rb) (300:\rs) };
	\end{scope}
	\draw[opacity=0.2] ({2/3},0) -- ({4/3},0)
	($0.5*(1,{1/sqrt(3)})$) -- ({2/3},0)
	($0.5*(1,{-1/sqrt(3)})$) -- ({2/3},0)
	($(1,{1/sqrt(3)})+0.5*(1,{-1/sqrt(3)})$) -- ({4/3},0)
	($0.5*(1,{1/sqrt(3)})+(1,{-1/sqrt(3)})$) -- ({4/3},0);
	\end{scope}
\begin{scope}[shift = ($(a)-(b)$)]
	\begin{scope}[xshift = 1.33333cm]
	\draw plot [smooth cycle] coordinates {(0:\rb) (60:\rs) (120:\rb) (180:\rs) (240:\rb) (300:\rs) };
	\end{scope}
	\begin{scope}[xshift = 0.666667cm, rotate=60]
	\draw plot [smooth cycle] coordinates {(0:\rb) (60:\rs) (120:\rb) (180:\rs) (240:\rb) (300:\rs) };
	\end{scope}
	\draw[opacity=0.2] ({2/3},0) -- ({4/3},0)
	($0.5*(1,{1/sqrt(3)})$) -- ({2/3},0)
	($0.5*(1,{-1/sqrt(3)})$) -- ({2/3},0)
	($(1,{1/sqrt(3)})+0.5*(1,{-1/sqrt(3)})$) -- ({4/3},0)
	($0.5*(1,{1/sqrt(3)})+(1,{-1/sqrt(3)})$) -- ({4/3},0);
	\end{scope}
	\begin{scope}[shift = ($-1*(a)+(b)$)]
	\begin{scope}[xshift = 1.33333cm]
	\draw plot [smooth cycle] coordinates {(0:\rb) (60:\rs) (120:\rb) (180:\rs) (240:\rb) (300:\rs) };
	\end{scope}
	\begin{scope}[xshift = 0.666667cm, rotate=60]
	\draw plot [smooth cycle] coordinates {(0:\rb) (60:\rs) (120:\rb) (180:\rs) (240:\rb) (300:\rs) };
	\end{scope}
	\draw[opacity=0.2] ({2/3},0) -- ({4/3},0)
	($0.5*(1,{1/sqrt(3)})$) -- ({2/3},0)
	($0.5*(1,{-1/sqrt(3)})$) -- ({2/3},0)
	($(1,{1/sqrt(3)})+0.5*(1,{-1/sqrt(3)})$) -- ({4/3},0)
	($0.5*(1,{1/sqrt(3)})+(1,{-1/sqrt(3)})$) -- ({4/3},0);
\end{scope}
\end{scope}

\draw[dashed,opacity=0.5,->] (-3.9,0.65) .. controls(-1.8,1.5) .. (1,0.7);
	\begin{scope}[scale=2.8]	
	\coordinate (a) at (1,{1/sqrt(3)});		
	\coordinate (b) at (1,{-1/sqrt(3)});	
	\coordinate (Y) at (1.8,0.45);
	\coordinate (c) at (2,0);
	\coordinate (x1) at ({2/3},0);
	\coordinate (x0) at (1,0);
	\coordinate (x2) at ({4/3},0);

	\pgfmathsetmacro{\rb}{0.25pt}
	\pgfmathsetmacro{\rs}{0.2pt}
	
	\begin{scope}[xshift = 1.33333cm]
	\draw plot [smooth cycle] coordinates {(0:\rb) (60:\rs) (120:\rb) (180:\rs) (240:\rb) (300:\rs) };
	\draw (0:\rb) node[xshift=7pt] {$D_2$};
	\end{scope}
	\begin{scope}[xshift = 0.666667cm, rotate=60]
	\draw plot [smooth cycle] coordinates {(0:\rb) (60:\rs) (120:\rb) (180:\rs) (240:\rb) (300:\rs) };
	\end{scope}
	\draw ({0.6666667-\rb},0) node[xshift=-7pt] {$D_1$};
	
	\draw (Y) node{$Y$};
	\draw[->] (0,0) -- (a) node[above]{$l_1$};
	\draw[->] (0,0) -- (b) node[pos=0.4,below left]{$Y_1$} node[below]{$l_2$};
	\draw[dashed] (a) -- (b) node[below]{$l_2$};
	\draw (a) -- (c) -- (b) node[pos=0.4,below right]{$Y_2$};
	\draw[fill] (x1) circle(0.5pt) node[xshift=6pt,yshift=-6pt]{$x_1$}; 
	\draw[fill] (x0) circle(0.5pt) node[yshift=4pt]{$x_0$}; 
	\draw[fill] (x2) circle(0.5pt) node[xshift=6pt,yshift=4pt]{$x_2$}; 
	\end{scope}
\end{tikzpicture}
	\caption{Illustration of the bubbly honeycomb crystal and quantities in the fundamental domain $Y$.} \label{fig:honeycomb}
\end{figure}
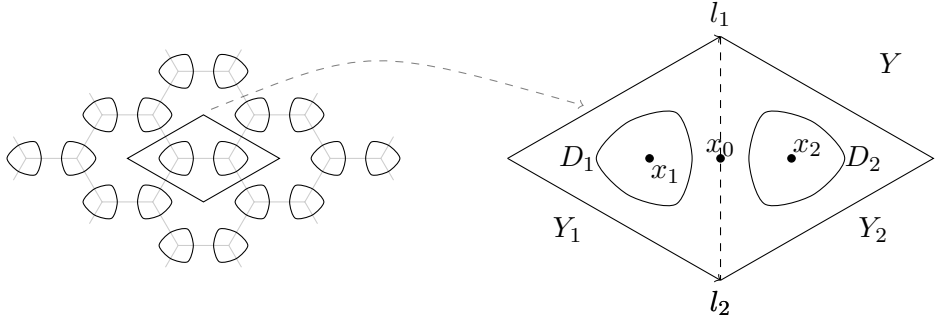

\hl{We assume that the bubbles are static and placed in a periodic crystal. Physically, this is achievable by placing the bubbles in a soft elastic medium} \cite{leroy1,leroy3}. We will assume that each bubble in the crystal has a three-fold rotational symmetry and that each pair of adjacent bubbles has a two-fold rotational symmetry. More precisely, let $R_1$ and $R_2$ be the rotations by $-\frac{2\pi}{3}$ around $x_1$ and $x_2$, respectively, and let $R_0$ be the rotation by $\pi$ around $x_0$. These rotations can be written as
$$ R_1 x = Rx+l_1, \quad R_2 x = Rx + 2l_1, \quad R_0 x = 2x_0 - x , $$
where $R$ is the rotation by  $-\frac{2\pi}{3}$ around the origin. Assume that each fundamental domain $Y_j, \ j=1,2$ contains one bubble $D_j$, which is a connected domain of Hölder class $\p D_j \in  C^{1,s}, \ 0<s<1$, satisfying 
$$ R_1 D_1 = D_1,\quad R_2 D_2 = D_2, \quad R_0 D_1 = D_2.$$
Denote the pair of bubbles, the bubble \emph{dimer}, by $D=D_1 \cup D_2$. Moreover, the full honeycomb crystal $\C$ is given by
$$\mathhl{\C = \bigcup_{m\in \Lambda} D + m.}$$ 

The dual lattice of $\Lambda$, denoted $\Lambda^*$, is generated by $\alpha_1$ and $\alpha_2$ satisfying $ \alpha_i\cdot l_j = 2\pi \delta_{ij}$,  for $i,j = 1,2.$ Then 
$$ \alpha_1 = \frac{2\pi}{a}\left( \frac{1}{\sqrt{3}}, 1\right),~~\alpha_2 = \frac{2\pi}{a}\left(\frac{1}{\sqrt{3}}, -1 \right).$$

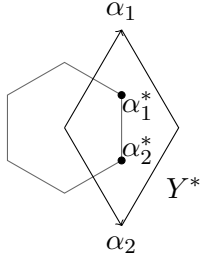
\begin{figure}
	\centering
	\begin{tikzpicture}[scale=1.3]	
	\coordinate (a) at ({1/sqrt(3)},1);		
	\coordinate (b) at ({1/sqrt(3)},-1);
	\coordinate (c) at ({2/sqrt(3)},0);
	\coordinate (K1) at ({1/sqrt(3)},{1/3});
	\coordinate (K2) at ({1/sqrt(3)},{-1/3});
	\coordinate (K3) at (0,{-2/3});
	\coordinate (K4) at ({-1/sqrt(3)},{-1/3});
	\coordinate (K5) at ({-1/sqrt(3)},{1/3});
	\coordinate (K6) at (0,{2/3});
	
	\draw[->] (0,0) -- (a) node[above]{$\alpha_1$};
	\draw[->] (0,0) -- (b) node[below]{$\alpha_2$};
	\draw (a) -- (c) -- (b) node[pos=0.4,below right]{$Y^*$};
	\draw[fill] (K1) circle(1pt) node[xshift=6pt,yshift=-4pt]{$\alpha_1^*$}; 
	\draw[fill] (K2) circle(1pt) node[xshift=6pt,yshift=4pt]{$\alpha_2^*$}; 
	%\draw[fill] (0,0) circle(1pt) node[left]{$\Gamma$}; 
	
	\draw[opacity=0.6] (K1) -- (K2) -- (K3) -- (K4) -- (K5) -- (K6) -- cycle; 
	\end{tikzpicture}
	\caption{Different representations of the Brillouin zone, generating vectors $\alpha_1, \alpha_2$, and the Dirac points $\alpha_1^*$, $\alpha_2^*$.} \label{fig:BZ}
\end{figure}
The Brillouin zone $Y^*$ is defined as the torus $Y^*:= {\R^2}/{\Lambda^*}$ and can be represented either as the unit cell
$$Y^* \simeq \left\{ s \alpha_1+ t \alpha_2 ~|~ 0 \le s,t \le 1 \right\}, $$
or as the \emph{first Brillouin zone}, which is the hexagon depicted in Figure \ref{fig:BZ}. \hl{As usual, for equivalence classes $\alpha, \beta \in Y^*$, with representatives $\alpha_0, \beta_0 \in \R^2$, we write $\alpha=\beta$ to denote $\alpha_0 = \beta_0 + q$ for some $q\in \Lambda^*$.}
The points $$\alpha_1^*= \frac{2\alpha_1+\alpha_2}{3}, \quad \alpha^*_2 = \frac{\alpha_1+2\alpha_2}{3} ,$$ in the Brillouin zone are called \emph{Dirac points}. Observe that, since $\Lambda^*$ is invariant under $R$, we have that $R:Y^*\rightarrow Y^*$ is a well-defined map. Moreover, we have 
$$
R\alpha_1^* = \frac{-\alpha_1+\alpha_2}{3} = \frac{2\alpha_1+\alpha_2}{3} = \alpha_1^*,$$
\hl{where the second equality follows due to the periodic nature of $Y^*$.} Similarly, $R\alpha_2^* = \alpha_2^*$.

Having defined the geometry, we now define the wave scattering problem in the bubbly honeycomb crystal. We denote by
$$\rho(x) = \rho_0\chi_{\R^2\setminus \C}(x) + \rho_b\chi_{\C}(x), \qquad \kappa(x) = \kappa_0\chi_{\R^2\setminus \C}(x) + \kappa_b\chi_{\C}(x),$$
where $\rho_0, \kappa_0$ and $\rho_b, \kappa_b$ denote the densities and bulk moduli \hl{outside and inside} the bubbles, respectively. Here, $\chi_{A}$ denotes the characteristic function of a set $A\in \R^2$. For a quasi-periodicity $\alpha \in Y^*$, we will study the $\alpha$-quasi-periodic Floquet component $u$ of the total wave field.  We therefore consider the following $\alpha$-quasi-periodic acoustic wave problem in $Y$:
\begin{equation}  \label{eq:pde}
\left\{
\begin{array} {lll}
\ds \nabla \cdot \frac{1}{\rho(x)} \nabla u(x)+ \frac{\omega^2}{\kappa(x)} u(x)  = 0 \quad &\text{in} \ \mathhl{\R^2}, \\[1em]
u(x+l)= e^{\iu\alpha\cdot l} u(x) \quad & \text{for all} \ l\in \Lambda.
\end{array}
\right.
\end{equation}
The differential equation in \eqnref{eq:pde} is the acoustic wave equation applied to time-harmonic waves of frequency $\omega$. The values of $\omega$ with positive real part such that there is a non-zero solution to \eqnref{eq:pde} are known as \emph{Bloch resonant frequencies}, or, seen as functions of $\alpha$, as \emph{band functions}. Since \eqnref{eq:pde} correspond to the Floquet transform of a self-adjoint operator, it is well-known that the band functions are real (see, for example, \cite[Chapter 5.2]{surv235} for an introduction to Floquet theory with applications to scattering problems). Let 
\begin{equation*} \label{v}
v:=\sqrt{\frac{\kappa_0}{\rho_0}}, \ ~  v_b:=\sqrt{\frac{\kappa_b}{\rho_b}}, \ ~k:=\frac{\omega}{v}, \ ~  k_b:=\frac{\omega}{ v_b}, \ ~ \delta:=\frac{\rho_b}{\rho_0}.
\end{equation*}
The parameter $\delta$ describes the contrast in the density and will be the key asymptotic parameter. We assume that $\rho_b$ is small compared to $\rho_0$ while the wave speeds are comparable and of order $1$, \ie{}, 
\begin{equation*} 
\delta  \ll 1 \quad \text{and} \quad v,  v_b=O(1).
\end{equation*}
In the case of air bubbles in water, $\delta \approx 10^{-3}$. The subwavelength frequency regime corresponds to frequencies $\omega$ which are considerably smaller than the lattice constant $a$. \hl{The motivation for studying systems of high-contrast bubbles is that such systems have resonant frequencies $\omega$ satisfying $\omega\rightarrow 0$ as $\delta\rightarrow 0$. In this work we say that a frequency $\omega$ is a subwavelength frequency if $\omega$ scales as $\omega = O(\delta^{1/2})$.}

With the notation as above, \eqnref{eq:pde} reads
\begin{equation}  \label{HP1}
\mathhl{\left\{
\begin{array} {lll}
&\ds \Delta  u+ k^2 u  = 0 \quad &\text{in} \ \R^2 \backslash \C, \\
\nm
&\ds \Delta u+ k_b^2 u  = 0 \quad &\text{in} \ \C, \\
\nm
&\ds  u_{+} -u_{-}  =0   \quad &\text{on} \ \p \C, \\
\nm
& \ds  \delta \frac{\partial u}{\partial \nu} \bigg|_{+} - \frac{\partial u}{\partial \nu} \bigg|_{-} =0 \quad &\text{on} \ \p \C, \\ 
& u(x+l)= e^{\iu\alpha\cdot l} u(x) \quad & \text{for all} \ l\in \Lambda.
\end{array}
\right.}
\end{equation}
Here, $\pd{}{\nu}$ denotes the normal derivative on $\partial \C$, and the subscripts $+$ and $-$ indicate the limits from outside and inside \hl{of} $\C$, respectively.

%%%%%%%%%%%%%%%%%%%%%%%%%%%%%%%%%%%%%%%%%%%%%%%%%%%%%%%%%%%%%%%%%%%%%%%%%%%%%%%%%%%%%
\subsection{Quasi-periodic Green's function for the honeycomb lattice}
In this section, we introduce the Green's function and the layer potentials that will be used in the sequel. A more detailed discussion can, for example, be found in \cite[Chapter 2.12]{surv235}.

Define the $\alpha$-quasi-periodic Green's function $G^{\alpha,k}$ to satisfy
\begin{equation*}
\Delta G^{\alpha,k} + k^2G^{\alpha,k} = \sum_{n \in \Lambda} \delta(x-n)e^{\iu\alpha\cdot n}.
\end{equation*}
\hl{If $k \neq |\alpha+q|$ for all $q\in \Lambda^*$, it can be shown} \cite{surv235,akl} that $G^{\alpha,k}$ is given by 
\begin{equation*}
G^{\alpha,k}(x)= \frac{1}{|Y|}\sum_{q\in \Lambda^*} \frac{e^{\iu(\alpha+q)\cdot x}}{ k^2-|\alpha+q|^2}.
\end{equation*}
From now on, we assume that the lattice constant $a$ is chosen so that $|Y|=1$.

For a given bounded domain $D$ in  $Y$, with Lipschitz boundary $\partial D$, the single layer potential of the density function $\varphi\in L^{2}(\partial D)$ is defined by
\begin{equation*}
\mathcal{S}_{D}^{\alpha,k}[\varphi](\Bx):=\int_{\partial D} G^{\alpha,k}( \Bx-\By ) \varphi(\By)\dx\sigma(\By),~~~\Bx\in \mathbb{R}^{2}.
\end{equation*}
The following jump relations are well-known \cite{surv235,akl}:
\begin{equation} \label{jump1}
\left. \frac{\partial}{\partial\nu} \mathcal{S}_{D}^{\alpha,k}[\varphi]\right|_{\pm}(\Bx)=\left (\pm\frac{1}{2}I+(\mathcal{K}_{D}^{-\alpha,k})^*\right)[\varphi](\Bx),~~~\Bx\in \partial D,
\end{equation}
where the Neumann-Poincar\'e operator $(\mathcal{K}_{D}^{-\alpha,k})^*$ is defined by 
\begin{equation*}
(\mathcal{K}_{D}^{-\alpha,k})^*[\varphi](x)=\text{p.v.}\int_{\partial D}\frac{\p}{\p \nu_x} G^{\alpha,k}(\Bx -\By) \varphi(\By)\dx\sigma(\By),~~~\Bx\in \partial D.
\end{equation*}
From \cite[Lemma 2.9]{surv235}, we know that $\Scal_D^{\alpha,\omega}: L^2(\partial D) \rightarrow H^1(\p D)$ is invertible when $\alpha \neq 0$ and for $\omega$ small enough.
 
In the rest of this work, we will assume that $\alpha$ is bounded away from zero. For some fixed $\alpha_0 >0, \alpha_0 \in \R$, we denote by $Y^*_0 = \{\alpha \in Y^* : |\alpha| \geq \alpha_0\}$. For $\alpha \in Y^*_0$ and $k$ small enough, we have $k \neq |\alpha+q|$ for all $q\in \Lambda^*$. Then, we can expand $G^{\alpha,k}$ with respect to $k$ as follows:
$$ G^{\alpha,k}= G^{\alpha,0}+ \sum_{j=1}^\infty k^{2j} G_j^{\alpha,0}, \quad \text{where} \quad G_j^{\alpha,0}(x):=-\sum_{q\in\Lambda^*}\mathhl{\frac{e^{\iu(\alpha+q)\cdot x}}{ |\alpha+q|^{2(j+1)}}}.$$
Using this expansion, we can expand the single layer potentials and the Neumann-Poincar\'e operators as
\begin{equation*}
\Scal_D^{\alpha,k} = \Scal_D^{\alpha,0} + \sum_{j = 1}^\infty k^{2j} \mathcal{S}^{\alpha,0}_{D,j}, \quad \text{where} \quad \mathcal{S}^{\alpha,0}_{D,j}[\phi](x):= \int_{\p D} G_j^{\alpha,0}(x-y)\phi(y)\dx\sigma(y),
\end{equation*}
and
\begin{equation*}
(\Kcal_D^{-\alpha,k})^* = (\Kcal_D^{-\alpha,0})^* + \sum_{j = 1}^\infty k^{2j} \mathcal{K}^{\alpha,0}_{D,j}, \quad \text{where} \quad \mathcal{K}^{\alpha,0}_{D,j} [\phi](x) := \int_{\p D} \frac{\p}{\p \nu_x} G_j^{\alpha,0}(x-y)\phi(y)\dx\sigma(y).
\end{equation*}
Since $G_j^{\alpha,0}$ is uniformly bounded for $x\in \R^2, \alpha \in Y^*_0$ and $j=1,2,..$, the operators $\Scal_{D,j}^{\alpha,0}$ and $\Kcal_{D,j}^{\alpha,0}$ are bounded operators for all $j$, and we have that $\|\Scal_{D,j}^{\alpha,0}\|_{\B(L^2(\p D), H^1(\p D))}$ and $\|\Kcal_{D,j}^{\alpha,0}\|_{\B(L^2(\p D), L^2(\p D))}$ are uniformly bounded for $\alpha \in Y^*, j=1,2,...$ Here, $\B(A,B)$ denotes the space of bounded operators between the normed spaces $A,B$ together with corresponding operator norm. We therefore have asymptotic expansions 
\begin{equation} \label{eq:SKexp}
\Scal_D^{\alpha,k} = \Scal_D^{\alpha,0} + k^2 \mathcal{S}^{\alpha,0}_{D,1} + O(k^4), \qquad 
(\Kcal_D^{-\alpha,k})^* = (\Kcal_D^{-\alpha,0})^* + k^{2} \mathcal{K}^{\alpha,0}_{D,1}+O(k^4),
\end{equation}
\hl{uniformly for $\alpha \in Y^*_0$, where the error terms are with respect to corresponding operator norm.} For $\alpha \in Y^*_0$, let $\psi_i\in L^2(\p D)$  be given by
\begin{align} \label{psi_def}
\psi_i = \left(\mathcal{S}_D^{\alpha,0}\right)^{-1}[\chi_{\p D_i}],\quad i=1, 2. 
\end{align}
In the following lemma, we collect some key properties of the layer potentials. The proof is analogous to proofs of similar results, in slightly different settings, found for example in \cite{surv235}.
\begin{lemma}\label{lem:layer}
	We assume $\alpha \in Y^*_0$.
	\begin{itemize}
		\item[(i)] We denote the $L^2(\p D)$-adjoint of the Neumann-Poincar\'e operator by $\Kcal_D^{\alpha,0}$. Then
		$$\ker\left(-\frac{1}{2}I + (\Kcal_D^{-\alpha,0})^*\right) = \mathrm{span}\{\psi_1, \psi_2\},\quad \ker\left(-\frac{1}{2}I + \Kcal_D^{\alpha,0}\right) = \mathrm{span}\{\chi_{\p D_1}, \chi_{\p D_2}\}.$$
		\item[(ii)] For any $\phi \in L^2(\p D)$ and for $j=1,2,$ we have		
		\begin{equation} \label{eq:Kint}
		\int_{\p D_j}\left(-\frac{1}{2}I +  (\Kcal_D^{-\alpha,0})^*\right)[\phi] \dx\sigma = 0.
		\end{equation}
		\item[(iii)] For any $\phi \in L^2(\p D)$ and for $j=1,2,$ we have		
		\begin{equation} \label{eq:K1int}
		\int_{\p D_j} \Kcal^{\alpha,0}_{D,1} [\phi](y) \dx\sigma(y) = -\int_{D_j}\Scal_D^{\alpha,0} [\phi](x) \dx x.
		\end{equation}
	\end{itemize}
\end{lemma}
\begin{proof}
	To prove (i), we first observe that the jump relation \eqnref{jump1} implies that $ \mathrm{span}\{\psi_1, \psi_2\} \subset \ker\left(-\frac{1}{2}I + (\Kcal_D^{-\alpha,0})^*\right)$. Conversely, if $\left(-\frac{1}{2}I + (\Kcal_D^{-\alpha,0})^*\right)[\psi] = 0$, we define $u(x) = \Scal_D^{\alpha,0}[\psi]$ and conclude from the jump relations that 
	$$
	\begin{cases}
	\Delta u =0 \quad &\mbox{in}~D,\\
	\ds \frac{\p u}{\p \nu} = 0 &\mbox{on}~\partial D.
	\end{cases}
	$$
	It follows that $u|_{D_i}$ is constant, so $u|_{\p D} = a_1\chi_{\p D_1} + a_2\chi_{\p D_2}$. Therefore $\psi = a_1\psi_1 + a_2\psi_2$, which proves $\ker\left(-\frac{1}{2}I + (\Kcal_D^{-\alpha,0})^*\right)  \subset \mathrm{span}\{\psi_1, \psi_2\} $. The second equality of (i) follows from the first, combined with the well-known Calder\'on identity $\Scal_D^{\alpha,0}(\Kcal_D^{-\alpha,0})^* =\Kcal_D^{\alpha,0}\Scal_D^{\alpha,0}$ \cite{surv235}.
	
	For the proof of (ii), we use the jump relation \eqnref{jump1} and integration by parts. Then
	$$\int_{\p D_j}\left(-\frac{1}{2}I +  (\Kcal_D^{-\alpha,0})^*\right)[\phi] \dx\sigma = \int_{\p D_j}\frac{\p \Scal_D^{\alpha,0}}{\p \nu}\Bigg|_-[\phi] \dx\sigma = \int_{D_j} \Delta \Scal_D^{\alpha,0}[\phi] \dx x = 0.$$
	To prove (iii), we use (ii) to conclude that on one hand
	$$\int_{\p D_j}\left(-\frac{1}{2}I +  (\Kcal_D^{-\alpha,k})^*\right)[\phi] \dx\sigma = 	k^2\int_{\p D_j} \Kcal^{\alpha,0}_{D,1} [\phi](y) \dx\sigma(y) + O(k^4\|\phi\|_{L^2(\p D)}).$$
	On the other hand, as in the proof of (ii) we have
	$$\int_{\p D_j}\left(-\frac{1}{2}I +  (\Kcal_D^{-\alpha,k})^*\right)[\phi] \dx\sigma = -k^2 \int_{D_j} \Scal_D^{\alpha,k}[\phi] \dx x = -k^2 \int_{D_j} \Scal_D^{\alpha,0}[\phi] \dx x + O(k^3\|\phi\|_{L^2(\p D)}),$$
	where we have used the expansion \eqnref{eq:SKexp}. Combined, we have
	$$
	\int_{\p D_j} \Kcal^{\alpha,0}_{D,1} [\phi](y) \dx\sigma(y) = -\int_{D_j}\Scal_D^{\alpha,0} [\phi](x) \dx x + O(k\|\phi\|_{L^2(\p D)}),
	$$
	and since the leading orders on the left-hand side and right-hand side are independent of $k$, they must coincide.
\end{proof}

Next, we derive asymptotic expansions for $\alpha$ near a Dirac point $\alpha^*$. From \cite{surv235}, we know that $G^{\alpha,k}(x) - G^{\alpha^*,k}(x)$ is continuously differentiable in $\alpha$ for $\alpha\in Y^*_0$, and bounded for $x\in Y$ and for $k$ in a neighbourhood of $0$. We therefore have the following asymptotic expansion of $G^{\alpha,k}$
\begin{align} \label{eq:exp}
G^{\alpha,k}(x) &= G^{\alpha^*,k}(x)+ \sum_{q\in\Lambda^*} \frac{e^{\iu(\alpha^*+q)\cdot x}}{ k^2-|\alpha^*+q|^2}  \left( \iu x \cdot (\alpha-\alpha^*) + 2\frac{(\alpha^*+q)\cdot(\alpha-\alpha^*)}{k^2-|\alpha^*+q|^2}\right) \nonumber\\
&  \qquad +O(|\alpha-\alpha^*|^2),
\end{align} 
uniformly for $k$ in a neighbourhood of $0$ and for $x\in Y$. We define $G_1^{k}$ by
$$ G_1^{k}(x):=\sum_{q\in\Lambda^*}\frac{e^{\iu(\alpha^*+q)\cdot x}}{ k^2-|\alpha^*+q|^2}  \left( \iu x + \frac{2(\alpha^*+q)}{k^2-|\alpha^*+q|^2}\right).$$
and the integral operators \hl{$\mathcal{S}^{k}_{1}$ and $\mathcal{K}_{1}^k$} as
\begin{align*}
\mathcal{S}^{k}_{1}[\phi](x):= \int_{\p D} G_1^{k}(x-y)\phi(y)\dx\sigma(y),\qquad \mathcal{K}_{1}^k [\phi](x):= \int_{\p D} \frac{\p}{\p \nu_x} G_1^k(x-y)\phi(y)\dx\sigma(y).
\end{align*}
We then have the expansions
\begin{align}
\Scal_D^{\alpha,k} &= \Scal_D^{\alpha^*,k} + \mathcal{S}^{k}_{1}\cdot (\alpha-\alpha^*) + O(|\alpha-\alpha^*|^2), \label{eq:S1}\\
(\Kcal_D^{-\alpha,k})^* &= (\Kcal_D^{-\alpha^*,k})^* + \Kcal^{k}_{1}\cdot (\alpha-\alpha^*) + O(|\alpha-\alpha^*|^2), \label{eq:K1}
\end{align}
\hl{uniformly for $k$ in a neighbourhood of $0$, where the error terms are with respect to the corresponding operator norm.}

%%%%%%%%%%%%%%%%%%%%%%%%%%%%%%%%%%%%%%%%%%%%%%%%%%%%%%%%%%%%%%%%%%%%%%%%%%%%%%%%%%%%%
 \subsection{Quasi-periodic capacitance matrix}\label{subsec:cap}
Let $V_i^\alpha, \ i=1,2,$ be the solution to
\begin{equation*} \label{cap_eq}
\begin{cases}
\Delta V_i^\alpha =0 \quad &\mbox{in}~\R^2\setminus \C,\\
V_i^\alpha = \delta_{ij} &\mbox{on}~\partial D_j,\\
V_i^\alpha(x+l)= e^{\iu\alpha\cdot l} V_i^\alpha(x) \quad &\forall l\in \Lambda.
\end{cases}
\end{equation*}
\hl{The intuitive idea for defining these functions is as follows. In the asymptotic limit $\delta \rightarrow 0$, the differential problem} \eqnref{HP1} \hl{decouples into a Neumann problem inside $D$ and a Dirichlet problem outside $D$. In the subwavelength limit $\omega \rightarrow 0$, this Neumann problem is solved by constant functions, and thus the Dirichlet data will be constant on $\p D$. Therefore, outside $D$, a solution $u$ to} \eqnref{HP1} \hl{can be approximated by a linear combination of $V_1^\alpha$ and $V_2^\alpha$. The fact that $u$ is approximately constant on $D$, and is thus (approximately) determined by these constant values, turns the continuous spectral problem} \eqnref{HP1} \hl{into a discrete eigenvalue problem in terms of the capacitance matrix $C^\alpha$ defined below. This underlying idea is made precise in the proof of Theorem} \ref{thm:delta}, \hl{in order to compute the Bloch resonant frequencies.}

We define the quasi-periodic capacitance coefficients $(C_{ij}^\alpha)$ by
$$ C_{ij}^\alpha := \int_{Y\setminus D} \overline{\nabla V_i^\alpha} \cdot \nabla V_j^\alpha \dx x,\quad i,j=1, 2.$$
\hl{Here, \emph{quasi-periodic} refers to the fact that  $(C_{ij}^\alpha)$ depend on the quasi-periodicity $\alpha\in Y^*$. In related work} \cite{H3a, ammari2017double}, \hl{analogous quantities, without the quasi-periodic assumption of $V_i^\alpha$, have been used to study the resonant frequencies of finite systems of resonators.} The quasi-periodic capacitance matrix $C^\alpha$ is defined as
$$C^\alpha =  \begin{pmatrix} C_{11}^\alpha & C_{12}^\alpha\\ C_{21}^\alpha & C_{22}^\alpha \end{pmatrix}.$$
The following lemma gives an equivalent description of $C_{ij}^\alpha$.
\begin{lemma}
For $\alpha \in Y^*_0$, the quasi-periodic capacitance coefficients $C_{ij}^\alpha$ are given by
\begin{equation}\label{eq:Cpsi}
C_{ij}^\alpha = - \int_{\partial D_i} \psi_j \dx \sigma,\quad i,j=1, 2,
\end{equation}
with $\psi_j$ as defined in \eqnref{psi_def}.
\end{lemma}
\begin{proof}
We will use the general fact that for quasi-periodic functions $v_1, v_2$ we have \cite[eq. (2.298)]{surv235}
\begin{equation} \label{eq:quasiint}
\int_{\p Y}  \overline{ \pd{v_1}{\nu}}v_2\dx\sigma = 0.
\end{equation}	
\hl{With $\psi_i$ as defined in} \eqnref{psi_def}, \hl{we have $V_i^\alpha = \Scal_D^{\alpha,0}[\psi_i]$ outside $\C$.} Then, using integration by parts, we have
$$ C_{ij}^\alpha = -\left(\int_{\p Y} \frac{\p V_j^\alpha }{\p \nu} \overline{V_i^\alpha}   \dx \sigma + \int_{\p D_j} \frac{\p V_j^\alpha }{\p \nu} \overline{V_i^\alpha}   \dx \sigma + \int_{\p D_i}\frac{\p V_j^\alpha }{\p \nu}  \overline{V_i^\alpha}  \dx \sigma\right).$$
The first integral vanishes due to \eqnref{eq:quasiint} while the second integral vanishes since $ \overline{V_i^\alpha} = 0$ on $\p D_j$. Then, since $ \overline{V_i^\alpha} = 1$ on $\p D_i$,
$$ C_{ij}^\alpha = - \int_{\p D_i}\frac{\p V_j^\alpha }{\p \nu} \dx \sigma.$$
From Lemma \ref{lem:layer} we have $\left(\frac{1}{2}I + (\Kcal_D^{\alpha,0})^*\right)[\psi_j] = \psi_j$. Then, using the jump relations \eqnref{jump1}, we find the desired expression for $C_{ij}^\alpha$.
\end{proof}

%%%%%%%%%%%%%%%%%%%%%%%%%%%%%%%%%%%%%%%%%%%%%%%%%%%%%%%%%%%%%%%%%%%%%%%%%%%%%%%%%%%%%%%%
\section{High-contrast subwavelength bands} \label{sec:pre}
In this section, we investigate the asymptotic behaviour of the band structure in the case of small $\delta$. As we shall see, this asymptotic limit enables explicit, approximate, computations of the band functions. The main results are given in Theorem \ref{thm:delta} and in equation (\ref{eq:asymptotic}). Throughout this section, we fix $ \alpha^*= \alpha^*_1 = \frac{2\alpha_1+\alpha_2}{3}$ and consider $\alpha \in Y^*_0$.

From \cite[Theorem 5.13]{surv235}, we know that the solution to \eqnref{HP1} can be represented using
the single layer potentials $\Scal_D^{\alpha, k_b}$ and $\Scal_D^{\alpha,k}$ as 
follows:
\begin{equation}\label{eq:rep}
u(\Bx) = \begin{cases}
\Scal_D^{\alpha,k_b} [\phi](\Bx),\quad \Bx\in D, \\
\Scal_D^{ \alpha,k}[\psi](\Bx),\quad \Bx\in Y\setminus\overline{D},\\
\end{cases}
\end{equation}
where, due to the jump conditions \eqnref{jump1}, the pair $(\phi,\psi)\in L^2(\p D)\times L^2(\p D)$ satisfies \beq\label{phipsi} \ \left\{
\begin{array}{l}
\ds \Scal^{\alpha, k_b}_D[\phi] - \Scal_D^{\alpha,k}[\psi] = 0\\
\nm
\ds \left (-\frac{1}{2}I+(\mathcal{K}_{ D}^{-\alpha, k_b})^*\right)[\phi]
-\delta\left (\frac{1}{2}I+(\mathcal{K}_{ D}^{-\alpha,k})^*\right)[\psi]=0
\end{array}\right.
\quad\mbox{on }\p D. \eeq  
We denote by
 \beq \label{eq:A}
 \Acal_\delta^{\alpha,\omega} := 
 \begin{pmatrix}
 \Scal^{\alpha, k_b}_D & -\Scal^{\alpha, k}_D\\
-\frac{1}{2}I+(\mathcal{K}_{D}^{-\alpha, k_b})^*  & -\delta \left (\frac{1}{2}I+(\mathcal{K}_{D}^{-\alpha,k})^*\right) 
 \end{pmatrix}.
 \eeq
We emphasise that $\Acal_{\delta}^{\alpha,\omega}$ is a function of $\omega$, since $k$ and $k_b$ depends on $\omega$. With this definition, \eqnref{HP1} is equivalent to
\begin{equation} \label{int_eq12}
\mathcal{A}_\delta^{\alpha,\omega} \begin{pmatrix} \phi \\ \psi \end{pmatrix}=0.
\end{equation}
It is well-known that the above integral equation has non-trivial solutions for some discrete frequencies $\omega$. 
These can be viewed as the characteristic values of the operator-valued analytic function $\mathcal{A}_\delta^{\alpha,\omega}$ (with respect to $\omega$); see \cite{surv235,akl} for the definition and properties of characteristic values. The Bloch resonant frequencies $\omega_j^\alpha$ are precisely the positive characteristic values. Moreover, from the original differential problem \eqnref{HP1}, we observe that the characteristic values are symmetric around the origin: if $\omega$ is a characteristic value, we have that $-\omega$ is also a characteristic value.

\subsection{Asymptotic computation of the band structure}
Here we compute the asymptotic band structure for a general $\alpha \in Y^*_0$, not necessarily close to a Dirac point.
\begin{lemma} \label{lem:first}
For $\alpha \in Y_0^*$, there are precisely two Bloch resonant frequencies $\omega_j^\alpha = \omega_j^\alpha(\delta), \ j=1,2,$ such that $\omega_j^\alpha(0)=0$ and $\omega_j^\alpha$ depends on $\delta$ continuously. 
\end{lemma}
\begin{proof}
In order to prove the lemma, we will apply the Gohberg-Sigal theory (see, for example, \cite{surv235, akl, Gohberg1971}). We will use the terminology and refer to the results presented in \cite[Chapter 1]{surv235}.

It is clear that $\omega=0$  is a characteristic value of 
$\Acal_0^{\alpha,\omega}$ because $$\Acal_0^{\alpha,0} = \begin{pmatrix}
 \Scal^{\alpha,0}_D & -\Scal^{\alpha, 0}_D\\
-\frac{1}{2}I+(\mathcal{K}_{D}^{-\alpha,0})^*  & 0
 \end{pmatrix}$$ has a non-trivial kernel of two dimensions, which is generated by 
\begin{equation} \label{eq:psi}
\Psi_1 = \begin{pmatrix}
\psi_1 \\
\psi_1
\end{pmatrix} 
\quad \text{and} \quad \Psi_2 = \begin{pmatrix}
\psi_2 \\
\psi_2
\end{pmatrix},
\end{equation}
where $\psi_1$ and $\psi_2$ are defined in \eqnref{psi_def}. From \eqnref{eq:SKexp}, we have that 
\begin{equation}\label{eq:sing}
\left(-\frac{1}{2}I+(\mathcal{K}_{D}^{-\alpha,k_b})^*\right)[\psi_i](x) = \omega^2 h(x,\omega), \quad x\in \p D, \quad i = 1,2,
\end{equation}
for some function $h$ which is holomorphic as a function of $\omega$ in a neighbourhood of $0$. Since
$$\int_{D_i} \Scal_D^{\alpha,0}[\psi_i](x) \dx x \neq 0, \ i=1,2,$$
it follows from \eqnref{eq:K1int} that $h(x,0)$ is not identically zero. Therefore, the rank of both $\Psi_1$ and $\Psi_2$ is 2. Since $\Psi_1$ and $\Psi_2$ are linearly independent, the multiplicity of the characteristic value $\omega=0$ is 4. 

The Neumann-Poincar\'e operator $(\Kcal_D^{-\alpha,0})^* : L^2(\partial D) \rightarrow L^2(\partial D)$ is well-known to be a compact operator \cite{surv235}, so $-\frac{1}{2}I+(\mathcal{K}_{D}^{-\alpha,0})^*$ is Fredholm of index zero. Since  $ \Scal^{\alpha,0}_D$ is invertible, also $\Acal_0^{\alpha,0}$ is Fredholm of index zero. Since $(\Kcal_D^{-\alpha,\omega})^*$ and $\Scal_D^{\alpha,\omega}$ are holomorphic as functions of $\omega$ in a neighbourhood of $0$, it follows that $\Acal_0^{\alpha,\omega}$ is of Fredholm type.

Let $V\subset \mathbb{C}$ be a disk around $0$ with a small enough radius, chosen such that $\Acal_0^{\alpha,\omega}$ is invertible on $\p V$ and $\omega = 0$ is the only characteristic value in $V$. From \cite[Lemma 1.11]{surv235}, it follows that $\Acal_0^{\alpha,\omega}$ is normal with respect to $\p V$. 

Now, we turn to the full operator $\Acal_\delta^{\alpha,\omega}$. It is clear that 
\begin{align*}
\Acal_\delta^{\alpha,\omega} =& \Acal_0^{\alpha,\omega} + \begin{pmatrix}0 & 0 \\ 0 & - \delta\left(\frac{1}{2}I + (\Kcal_D^{-\alpha,k})^*\right)\end{pmatrix} \\
:=&  \Acal_0^{\alpha,\omega} + A^{(1)}(\omega,\delta),
\end{align*}
where $ A^{(1)}(\omega,\delta)$, as a function of $\omega$, is holomorphic in $V$ and continuous up to $\partial V$. For small enough $\delta$ we have
$$\left\|\left(\Acal_0^{\alpha,\omega}\right)^{-1}A^{(1)}(\omega,\delta) \right\|_{\B\big((L^2(\p D))^2,(L^2(\p D))^2\big)} < 1, \quad \omega \in \p V.$$
Hence, the generalization of Rouch\'e's theorem \cite[Theorem 1.15]{surv235} shows that $\Acal_\delta^{\alpha,\omega}$ has 4 characteristic values inside $V$, for small enough $\delta$. Since the characteristic values are symmetric around the origin, is clear that two of these, namely $\omega_1^\alpha$ and $\omega_2^\alpha$, have positive real parts, while two characteristic values have negative real parts.

The fact that $\omega_1^\alpha(\delta) $ and $\omega_2^\alpha(\delta) $ are continuous in $\delta$ can be deduced in a similar way: if $U \in \mathbb{C}$ is a neighbourhood of $\omega_i^\alpha(\delta_1), \ i=1,2$, we can write  
\begin{align*}
\Acal_{\delta_2}^{\alpha,\omega} = \Acal_{\delta_1}^{\alpha,\omega} + (\delta_1-\delta_2) \begin{pmatrix}0 & 0 \\ 0 & \frac{1}{2}I + (\Kcal_D^{-\alpha,k})^*\end{pmatrix},
\end{align*}
and from the generalization of Rouch\'e's theorem it follows that $\omega^\alpha_i(\delta_2)\in U$ when $|\delta_1-\delta_2|$ is small enough.
\end{proof}

\begin{theorem} \label{thm:delta}
The band functions $\omega_j^\alpha=\omega_j^\alpha(\delta),~j=1,2$ of $\mathcal{A}_\delta^{\alpha,\omega}$ can be approximated as
$$ \omega_j^\alpha= \sqrt{\frac{\delta \lambda_j^\alpha }{|D_1|}}  v_b + O(\delta),$$
uniformly for $\alpha \in Y^*_0$, where $|D_1|$ is the volume of one resonator and $\lambda_j^\alpha,~j=1,2$ are the eigenvalues of the quasi-periodic capacitance matrix $C^\alpha$.
\end{theorem}
\pf \hl{We seek solutions $(\phi, \psi)$ to the integral equation} \eqnref{int_eq12}, \hl{normalized such that $\|\phi\|_{L^2(\p D)} = 1$ and $\|\psi\|_{L^2(\p D)} = 1$. Using the asymptotic expansions} \eqnref{eq:SKexp}, \hl{we find that $\phi$ and $\psi$ satisfy}
\begin{align}
\Scal_D^{\alpha,0}[\phi]- \Scal_D^{\alpha, 0}[\psi] &= O(\omega^2), \nonumber \\
\left(-\frac{1}{2}I + (\Kcal_D^{-\alpha,0})^* + k_b^2 \Kcal_{D,1}^{\alpha,0}\right)[\phi]-\delta \left(\frac{1}{2}I+(\mathcal{K}_{D}^{\alpha,0})^*\right)[\psi]&=O(\omega^4 +\delta^2),\label{eq_line2}
\end{align}
\hl{uniformly for $\alpha \in Y^*_0$, where the error terms are with respect to the norm in $L^2(\p D)$. Observe that $\Scal_D^{\alpha,0}$ is invertible, and by the inverse mapping theorem together with the fact that $Y^*_0$ is a closed set we have that $\|(\Scal_D^{\alpha,0})^{-1}\|_{\B(L^2(\p D),H^1(\p D))}$ is uniformly bounded for $\alpha \in Y^*_0$.} Then, we get
\begin{equation*} 
\psi= \phi+O(\omega^2),
\end{equation*}
uniformly for $\alpha \in Y^*_0$. Inserting the above approximation into \eqnref{eq_line2}, we obtain that
\begin{equation}
\left(-\frac{1}{2}I + (\Kcal_D^{-\alpha,0})^* + k_b^2 \Kcal_{D,1}^{\alpha,0}\right)[\phi]-\delta \left(\frac{1}{2}I+(\mathcal{K}_{D}^{\alpha,0})^*\right)[\phi]=
O(\omega^4 +\delta^2),  \label{int_eq_redc}
\end{equation}
uniformly for $\alpha \in Y^*_0$. Recall that 
$ \ker \left(-\frac{1}{2}I + (\Kcal_D^{-\alpha,0})^*\right)$ is spanned by $\psi_1$ and $\psi_2$. Then we write $\phi$ as
\begin{equation}
\mathhl{\phi= a \psi_1 + b\psi_2 + \varphi,} \label{phi_approx}
\end{equation}
where $\varphi$ is orthogonal to $\mathrm{span}\{ \psi_1,\psi_2\}$ in $L^2(\p D)$. We then have from \eqnref{int_eq_redc} that
$$
\left(-\frac{1}{2}I + (\Kcal_D^{-\alpha,0})^*\right)[\varphi] = O(\omega^2 + \delta).
$$
uniformly for $\alpha \in Y^*_0$. Since $\left(-\frac{1}{2}I + (\Kcal_D^{-\alpha,0})^*\right)$, restricted to the orthogonal complement of $\mathrm{span}\{ \psi_1,\psi_2\}$, is invertible with bounded inverse, it follows that, in the $L^2(\p D)$-norm,
$$
\varphi = O(\omega^2 + \delta),
$$
uniformly for $\alpha \in Y^*_0$. Moreover, find that $|a| + |b| > 0$. Now, we substitute \eqnref{phi_approx} into \eqnref{int_eq_redc} and integrate around $\partial D_i$ for $i = 1,2$. Then, using \eqnref{eq:K1int}, we get
\begin{align*}
-\frac{\omega^2 |D_1|}{v_b^2} a +\delta (a C_{11}^\alpha+b C_{12}^\alpha)= O(\omega^4+\delta^2),\\
-\frac{\omega^2 |D_1|}{v_b^2} b +\delta (a C_{21}^\alpha+b C_{22}^\alpha)= O(\omega^4+\delta^2),
\end{align*}
uniformly for $\alpha \in Y^*_0$. Therefore, $\frac{\omega^2 |D|}{\delta v_b^2}$ approximates the eigenvalues of the quasi-periodic capacitance matrix. This completes the proof. 
\qed

\subsection{Asymptotic band structure close to Dirac points}
Theorem \ref{thm:delta} gives an asymptotic formula of the band functions in terms of $\delta$. In this section, we will investigate the behaviour of this approximation for $\alpha$ close to the Dirac points. \hl{For $\alpha \in Y^*$, we define two transformations $T_1^\alpha$ and $T_2$ of $\alpha$-quasi-periodic functions in $Y$ by}
\begin{equation*}\label{t1t2}
(T_1^\alpha f)(x):= \begin{cases} e^{-\iu\alpha\cdot l_1} f(R_1x), &x\in Y_1,\\
e^{-2\iu\alpha\cdot l_1} f(R_2x), &x\in Y_2, \end{cases} \qquad (T_2f)(x):= \overline{f(2x_0 - x)}.
\end{equation*} 
Observe that $T_1^\alpha f$ is well-defined on $\partial Y_1 \cap \partial Y_2$. \hl{For any $\alpha \in Y^*$, we have that $T_2f$ is $\alpha$-quasi-periodic, while at $\alpha = \alpha^*$ we have that $T_1^{\alpha^*} f$ is $\alpha^*$-quasi-periodic.} We will denote $T_1 := T_1^{\alpha^*}$, and define  $\tau:= e^{2\pi \iu /3}$. 

We remark that the quasi-periodic capacitance matrix $C^\alpha$ is Hermitian for any $\alpha$. At Dirac points, the following result holds.
\begin{lemma}
	At the Dirac points, the quasi-periodic capacitance matrix is a constant multiple of the identity matrix.
\end{lemma}
\pf Since $T_2V_1^{\alpha}=V_2^{\alpha}$, we have 
$$C_{11}^{\alpha}=\int_{Y\setminus D} \overline{\nabla V_1^\alpha}\cdot \nabla V_1^\alpha \dx x = \int_{Y\setminus D} \overline{\nabla T_2 V_1^\alpha}\cdot \nabla T_2V_1^{\alpha} \dx x =C_{22}^{\alpha}.$$
Hence $C_{11}^{\alpha}=C_{22}^{\alpha}$ for any $\alpha \in Y^*_0$ and in particular $C_{11}^{\alpha^*}=C_{22}^{\alpha^*}$. We can also check that
$$ T_1V_1^{\alpha^*}= \tau V_1^{\alpha^*},~T_1 V_2^{\alpha^*} = \tau^2 V_2^{\alpha^*}.$$
Then, it follows that
$$ C_{12}^{\alpha^*}=\int_{Y\setminus D}\overline{\nabla V_1^{\alpha^*}} \cdot \nabla V_2^{\alpha^*} \dx x = \int_{Y\setminus D} \overline{\nabla T_1V_1^{\alpha^*}} \cdot \nabla T_1V_2^{\alpha^*}\dx x= \tau C_{12}^{\alpha^*}.$$
Therefore, we have $C_{12}^{\alpha^*}=0$ and so, $C_{21}^{\alpha^*}=0$. 
\qed

Since the quasi-periodic capacitance matrix has a double eigenvalue at Dirac points, we have that $\omega_1^{\alpha^*}=\omega_2^{\alpha^*} + O(\delta)$. The following proposition shows that in fact $\omega_1^{\alpha^*} = \omega_2^{\alpha^*}$ and that this is a double characteristic value.
\begin{prop} \label{prop_mul}
At the Dirac point $\alpha=\alpha^*$ and for $\delta$ small enough, the first Bloch resonant frequency $\omega^*:=\omega_1^{\alpha^*}$ is of multiplicity $2$, \ie{},  $\mathcal{A}_\delta^{\alpha^*,\omega^*}$ has a two dimensional kernel.
\end{prop}

\pf From Lemma \ref{lem:first} we know that there are only two band functions (counted with multiplicity) converging to $0$ as $\delta\rightarrow 0$. Hence, for small enough $\delta$, the dimension of $\ker\left(\mathcal{A}_\delta^{\alpha^*,\omega^*} \right)$ is at most $2$. Suppose that $\omega^*$ is of multiplicity $1$. Suppose also that
$$ \mathcal{A}_\delta^{\alpha^*,\omega^*} \begin{pmatrix} \phi \\ \psi \end{pmatrix}=0,$$
for a non-trivial pair $(\phi,\psi)$.
\hl{Denote}  $k^* = \omega^*/v$ and $k_b^* = \omega^*/v_b$ and let
\begin{equation*}
u(\Bx) = \begin{cases}
 \Scal_D^{\alpha^*,k_b^*} [\phi](\Bx),\quad \Bx\in D, \\
 \Scal_D^{ \alpha^*,k^*}[\psi](\Bx),\quad \Bx\in Y\setminus\overline{D}.\\
 \end{cases}
\end{equation*}

We can easily check that $T_1 u$ and $T_2 u$ also satisfy \eqnref{HP1}.
Then  
 $u,~T_1 u,$ and $T_2 u$
are linearly dependent and so, 
$$ T_1 u =c_1 u, ~T_2 u = c_2 u,$$
for some non-zero constants $c_1$ and $c_2$. 
Here, we observe that
$$T_1^3=I, ~T_2^2=I,~ T_1 T_2 f = T_2 T_1 f. $$ Then, it follows that 
 $ c_1 \in \{1, \tau, \tau^2\}$, $c_2=\pm 1$, and
$$  c_1c_2 u=T_1T_2 u = T_2 T_1 u =  \bar c_1 c_2 u.$$
Therefore, we get $c_1=1$, \ie{}, $T_1u = u$. However, taking the expression for $\phi$ in \eqnref{phi_approx}, along with the representation \eqnref{eq:rep}, we find that $u$ is constant on $D_i, i=1,2$ up to an error of order $O(\delta)$. This contradicts $T_1u = u$, which completes the proof. 
\qed
 
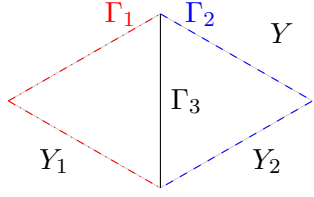
\begin{figure}[htb]
	\centering
	\begin{tikzpicture}[scale=2]
	\coordinate (a) at (1,{1/sqrt(3)});		
	\coordinate (b) at (1,{-1/sqrt(3)});	
	\coordinate (Y) at (1.8,0.45);
	\coordinate (c) at (2,0);
	\coordinate (x1) at ({2/3},0);
	\coordinate (x0) at (1,0);
	\coordinate (x2) at ({4/3},0);
	
	\draw (Y) node{$Y$};
	\draw[opacity=0.2] (0,0) -- (a);
	\draw[opacity=0.2] (0,0) -- (b);
	\draw[opacity=0.2] (a) -- (c) -- (b);
	
	\draw[red,dash dot] (0,0) -- (a);
	\draw[red,dash dot] (0,0) -- (b);
	\draw (a) -- (b) node[right,pos=0.5]{$\Gamma_3$};
	\draw[blue,dashed] (a) -- (c) -- (b);
	\draw (0.3,-0.4) node{$Y_1$};
	\draw (1.7,-0.4) node{$Y_2$};
	\draw[blue] (a) node[xshift=15pt]{$\Gamma_2$};
	\draw[red] (a) node[xshift=-15pt]{$\Gamma_1$};
	\end{tikzpicture}
	\caption{Fundamental domain $Y$ and the curves $\Gamma_1, \Gamma_2$ and $\Gamma_3$ used in the proof of Lemma \ref{lem:c}.} \label{fig:gamma}
\end{figure} 
We define the curves $\Gamma_1, \Gamma_2$ and $\Gamma_3$ illustrated in Figure \ref{fig:gamma} as 
$$ \Gamma_1:=\p Y_1 \cap \p Y, \quad \Gamma_2 := \p Y_2 \cap \p Y, \quad \Gamma_3:= \p Y_1\cap \p Y_2.$$
In the sequel, we will occasionally need the following additional assumption on the bubble geometry.
\begin{asump} \label{asump:g3}
	$D$ is symmetric with respect to $\Gamma_3$, \ie{} $D$ satisfies
	$R_3D = D$, 
	where $R_3$ is the reflection across $\Gamma_3$.
\end{asump}
\hl{In the next lemma, and throughout the remainder of this work, we will use bracketed subscripts to denote components of vectors; as an example we write} $\alpha = \left( \begin{smallmatrix} \alpha_{(1)} \\ \alpha_{(2)} \end{smallmatrix} \right)$ for $\alpha \in Y^*$.
\begin{lemma} \label{lem:c} The quasi-periodic capacitance matrix coefficients 
	$C_{11}^\alpha $ and $C_{12}^\alpha $ are differentiable with respect to $\alpha$ at $\alpha=\alpha^*$. Moreover,
	$$\nabla_\alpha C_{11}^\alpha \Big|_{\alpha=\alpha^*}=0,\quad \nabla_\alpha C_{12}^\alpha \Big|_{\alpha=\alpha^*} = c \begin{pmatrix}1\\ -\iu\end{pmatrix}, $$
	where $c := \pd{C_{12}^\alpha}{\alpha_{(1)}}\Big|_{\alpha=\alpha^*}$. Under Assumption \ref{asump:g3}, we have $ c \ne 0$.
\end{lemma}
\begin{proof}
	For a small $\epsilon \in Y^*$, we have from \eqnref{eq:S1}
	\begin{equation*}
	\mathhl{\left(\Scal_D^{\alpha^*+\epsilon,0}\right)^{-1} = \left(\Scal_D^{\alpha^*,0}\right)^{-1} + \left(\Scal_D^{\alpha^*,0}\right)^{-1} \left(\Scal_{1}^0\cdot \epsilon\right)\left(\Scal_D^{\alpha^*,0}\right)^{-1} + O(|\epsilon|^2),}
	\end{equation*}
	where the error term is with respect to the operator norm. Therefore, from \eqnref{eq:Cpsi} it follows that the quasi-periodic capacitance coefficients are differentiable at $\alpha=\alpha^*$.
	
	We have the relations
	$$  V_1^{R^2\alpha}(x) =  \begin{cases}  V_1^\alpha(R_1x),\quad &x\in Y_1,\\
	e^{-\iu\alpha \cdot l_1} V_1^\alpha(R_2x),\quad &x\in Y_2, \end{cases}$$
	$$  V_2^{R^2\alpha}(x) =  \begin{cases} e^{\iu\alpha \cdot l_1} V_2^\alpha(R_1x),\quad &x\in Y_1,\\
	V_2^\alpha(R_2x),\quad &x\in Y_2, \end{cases}$$
	from which it follows that
	$$ C_{11}^\alpha= C_{11}^{R\alpha}= C_{11}^{R^2\alpha}, \quad C_{12}^{R^2\alpha} = e^{\iu\alpha \cdot l_1} C_{12}^\alpha.$$
	Differentiating these expressions, and applying $R\alpha^* = \alpha^*$, we arrive  at
	$$ \nabla_\alpha C_{11}^\alpha \Big|_{\alpha=\alpha^*}=0,\quad \nabla_\alpha C_{12}^\alpha \Big|_{\alpha=\alpha^*} = c \begin{pmatrix}1\\-\iu \end{pmatrix}, $$
	where $c = \pd{C_{12}^\alpha}{\alpha_{(1)}}\Big|_{\alpha=\alpha^*}$.  It only remains to show that $c \ne 0$ if $D$ is symmetric with respect to  $\Gamma_3$. Let $\hat V_j := \pd{V_j^\alpha}{\alpha_{(1)}}\big|_{\alpha=\alpha^*}$, which satisfies
\begin{equation*}
\begin{cases}
\Delta \hat V_j =0 \quad &\mbox{in}\quad Y\setminus \overline{D},\\
\hat V_j =0 &\mbox{on}\quad \p D,\\
\hat V_j( x + l) = e^{\iu\alpha\cdot l} \hat V_j(x) +\iu l_{(1)} e^{\iu\alpha\cdot l} V_j^\alpha (x), \quad  & l\in \Lambda,
\end{cases}
\end{equation*}
where $l=\left( \begin{smallmatrix} l_{(1)} \\ l_{(2)}\end{smallmatrix}\right)$. Then, using quasi-periodicity of $V_j^\alpha$ and $\hat V_j$, we have
\begin{align}
\pd{C_{12}^{\alpha}}{\alpha_{(1)}} \Big|_{\alpha=\alpha^*} &= \int_{Y\setminus D} \left(\overline{ \nabla V_1^{\alpha^*}} \cdot  \nabla \hat V_2 + \overline{\nabla  \hat V_1} \cdot  \nabla  V_2^{\alpha^*} \right) \dx x \nonumber \\
&= \int_{\p Y} \left(\overline{ \pd{V_1^{\alpha^*}}{\nu} }    \hat V_2+   \overline{ \hat V_1} \pd{V_2^{\alpha^*}}{\nu}   \right)\dx \sigma \nonumber \\
&=\iu\frac{\sqrt{3}a}{2}\int_{\Gamma_2} \left( \overline{ \pd{V_1^{\alpha^*}}{\nu}}   V_2^{\alpha^*} -\overline{ V_1^{\alpha^*}} \pd{V_2^{\alpha^*}}{\nu}    \right)\dx \sigma, \label{eq:pc}
\end{align}
where we have used the fact that $\frac{\sqrt{3}a}{2}$ is the $l_{(1)}$-component of the lattice vectors $l_1$ and $l_2$. Using \eqnref{eq:quasiint}, we find that
\begin{align}
\int_{\Gamma_1}  \overline{ \pd{V_1^{\alpha^*}}{\nu}}   V_2^{\alpha^*}\dx\sigma =- \int_{\Gamma_2} \overline{ \pd{V_1^{\alpha^*}}{\nu}}   V_2^{\alpha^*}\dx\sigma =- \int_{\Gamma_2} \overline{ \pd{(T_2V_2^{\alpha^*})}{\nu} }  T_2V_1^{\alpha^*} \dx \sigma
= - \int_{\Gamma_1} \overline{ V_1^{\alpha^*} }\pd{V_2^{\alpha^*}}{\nu}   \dx \sigma. \label{eq:V1V2}
\end{align}
Assume that $ \pd{C_{12}^\alpha}{\alpha_{(1)}}\Big|_{\alpha=\alpha^*} = 0$. Then, from \eqnref{eq:pc} and \eqnref{eq:V1V2}, we obtain that
\begin{equation}\label{gamma_zero}
\int_{\Gamma_j} \overline{ V_1^{\alpha^*}} \pd{V_2^{\alpha^*}}{\nu} \dx\sigma = \int_{\Gamma_j}  \overline{ \pd{V_1^{\alpha^*}}{\nu}}  V_2^{\alpha^*}\dx\sigma=0, \end{equation}
for $j=1,2$. Let $\nu_3 = (1,0)$. Then it follows from \eqnref{gamma_zero}  and $C_{12}^{\alpha^*}=0$ that
\begin{equation}\label{gamma3_eval}
\int_{\Gamma_3}  \left( \overline{\pd{V_1^{\alpha^*}}{\nu_3}} V_2^{\alpha^*}-\overline{V_1^{\alpha^*}}\pd{V_2^{\alpha^*}}{\nu_3}\right)\dx\sigma = \int_{Y\setminus D} \overline{\nabla V_1^{\alpha^*}} \cdot \nabla V_2^{\alpha^*} \dx x=0.
\end{equation}
If we write $\alpha^* = \left( \begin{smallmatrix} \alpha_{(1)}^* \\ \alpha_{(2)}^*\end{smallmatrix}\right)$, then $\left( \begin{smallmatrix} -\alpha_{(1)}^* \\ \alpha_{(2)}^*\end{smallmatrix}\right)=\alpha^*$ in $Y^*$. Under Assumption \ref{asump:g3}, \ie{} if $D$ is symmetric with respect to $\Gamma_3$, we observe that $V_1^{\alpha^*}(x) = V_2^{(-\alpha_{(1)}^*, \alpha_{(2)}^*)}(R_3x) = V_2^{\alpha^*}(R_3x)$. From this, we have that 
$$V_1^{\alpha^*} = V_2^{\alpha^*}, \quad \pd{V_2^{\alpha^*}}{\nu_3} = - \pd{V_1^{\alpha^*}}{\nu_3} \quad \text{on } \Gamma_3.$$
Together with \eqnref{gamma3_eval}, we get
\begin{equation}\label{gamma3_eval2}
2 \mbox{Re}\int_{\Gamma_3}    V_1^{\alpha^*}\overline{ \pd{V_1^{\alpha^*}}{\nu_3}}\dx\sigma=  \int_{\Gamma_3}  \left( \overline{\pd{V_1^{\alpha^*}}{\nu_3}} V_2^{\alpha^*} - \overline{V_1^{\alpha^*}}\pd{V_2^{\alpha^*}}{\nu_3}\right)\dx\sigma =0.
\end{equation}
We recall that $T_1 V_1^{\alpha^*} = \tau V_1^{\alpha^*}$ and observe that $\Gamma_1 = \big(R_1 \Gamma_3 \big)\cup \big((R_1)^2\Gamma_3\big)$. Using these facts and \eqnref{gamma3_eval2}, we get 
$$ \mbox{Re} \int_{\Gamma_1}  V_1^{\alpha^*}\overline{ \pd{V_1^{\alpha^*}}{\nu}}\dx\sigma=0,$$
so that $$ \int_{Y_2\setminus D_2} |\nabla V_1^{\alpha^*}|^2 \dx x=0.$$
Combined with $ V_1^{\alpha^*}\Big|_{\p D_2}=0$, this tells us that $V_1^{\alpha^*} =0$ in $Y\setminus D$, which is a contradiction. Therefore,  we can conclude that $ \pd{C_{12}^\alpha}{\alpha_{(1)}}\Big|_{\alpha=\alpha^*} = c \ne 0$. 
\end{proof}

Starting from the asymptotic formula for $\omega_j^\alpha$ in terms of $\delta$, we can now ascertain the \hl{asymptotic} dependence for $\alpha$ close to $\alpha^*$. From Lemma \ref{lem:c}, we obtain that
\begin{align*}
C_{11}^\alpha &= C_{11}^{\alpha^*} + O(|\alpha-\alpha^*|^2),\quad |C_{12}^\alpha|= \left|\pd{C_{12}^\alpha}{\alpha_{(1)}}\Big|_{\alpha=\alpha^*}\right| |\alpha-\alpha^*| + O(|\alpha-\alpha^*|^2).
\end{align*}
Because $C^\alpha$ is Hermitian with identical diagonal elements, the eigenvalues are given by $\lambda_j^\alpha = C_{11}^\alpha \pm |C_{12}^\alpha|$. For $\alpha$ close to $\alpha^*$,  we find the following asymptotic behaviour:
\begin{align*} 
\lambda_j^\alpha = C_{11}^{\alpha^*}\pm \left|\pd{C_{12}^\alpha}{\alpha_{(1)}}\Big|_{\alpha=\alpha^*}\right| |\alpha-\alpha^*| + O(|\alpha-\alpha^*|^2).
\end{align*}
Now, we conclude that $\lambda_1^{\alpha^*}=\lambda_2^{\alpha^*}=C_{11}^{\alpha^*}$ and
\begin{equation} \label{eq:asymptotic}
\omega_j^\alpha(\delta)= \sqrt{\frac{\delta C_{11}^{\alpha^*} }{|D|}}  v_b \left( 1\pm \frac{\left|\pd{C_{12}^\alpha}{\alpha_{(1)}}\Big|_{\alpha=\alpha^*}\right|}{2C_{11}^{\alpha^*}}|\alpha-\alpha^*|+O(|\alpha-\alpha^*|^2) \right)+ O(\delta).
\end{equation}
Equation \eqnref{eq:asymptotic} gives the asymptotic band structure for small $\delta$, and suggests that the system has a Dirac cone. However, we do not know the behaviour of the error term $O(\delta)$, so we can not conclude the existence of a Dirac cone from equation \eqnref{eq:asymptotic} alone. This will be addressed in the following section. 

\begin{remark} Theorem \ref{thm:delta} shows that \hl{$(\omega_j^\alpha)^2$} scales like $O(\delta)$ for small $\delta$. In \cite{H3a}, it was found the Minnaert resonant frequency $\omega_M$ of a single bubble in free space scales according to $\omega_M^2\ln \omega_M = O(\delta)$ in two dimensions. 
Thus, $\omega_j^\alpha$ has a different asymptotic behaviour than $\omega_M$. The difference is explained by the quasi-periodic single layer potential not exhibiting a $\log$-singularity as $\omega\rightarrow 0$.
\end{remark}

\section{Conical behaviour of subwavelength bands at Dirac points} \label{sec:dirac}
In this section, we prove the main result of the paper, Theorem \ref{main}. \hl{In contrast to the approximations in Section} \ref{sec:pre},\hl{ we here prove the existence of an \emph{exact} Dirac cone. However, unlike Section} \ref{sec:pre}, \hl{the method utilized here does not enable explicit computations of the slope and centre frequency of the cone.} As before, let $\omega^*$ be the Bloch resonant frequency at $\alpha^*=\frac{2\alpha_1+\alpha_2}{3}$.
\begin{theorem}\label{main}
	For sufficiently small $\delta$, the first and second band functions form a Dirac cone at $\alpha^*$, \ie{},
\begin{align} 
\omega_1^\alpha(\delta) = \omega^*- \lambda |\alpha - \alpha^*| \big[ 1+ O(|\alpha-\alpha^*|) \big],  \label{eq:cone1}\\
\omega_2^\alpha(\delta) = \omega^*+ \lambda |\alpha - \alpha^*| \big[ 1+ O(|\alpha-\alpha^*|) \big],  \label{eq:cone2}
\end{align}
for some $\lambda$ independent of $\alpha$, where the error term $O(|\alpha-\alpha^*|)$ is uniform in $\delta$. As $\delta\rightarrow 0$, we have the following asymptotic expansions of $\omega^*$ and $\lambda$:
$$ \omega^*= \sqrt{\frac{\delta C_{11}^{\alpha^*} }{|D|}}  v_b + O(\delta), \qquad \lambda=\frac{1}{2}\sqrt{\frac{\delta }{|D|C_{11}^{\alpha^*}}}  v_b \left|\pd{C_{12}^\alpha}{\alpha_{(1)}}\Big|_{\alpha=\alpha^*}\right| + O(\delta).$$
Under Assumption \ref{asump:g3}, $\lambda$ is non-zero for sufficiently small $\delta$. 
\end{theorem}
\hl{Unlike the method in Section} \ref{sec:pre}, \hl{where the operator $\mathcal{A}_\delta^{\alpha,k}$ was asymptotically expanded in terms of $\omega$ and $\delta$, the idea is now to expand this operator for $\alpha$ close to $\alpha^*$, while keeping the dependence on $\omega$ and $\delta$ exact.} In Section \ref{sec:lemmas} we prove some preliminary results, while the main proof of Theorem \ref{main} is given in Section \ref{sec:pf}.

\subsection{Preliminary lemmas} \label{sec:lemmas}
Here, we prove Lemmas \ref{lemma_t1t2}, \ref{ker_basis}, \ref{rel_st} and \ref{lem:err} needed for the proof of Theorem \ref{main}. In the following, we will interchangeably use $T_1, T_2$ and $T_1^\alpha$ as operators on $L^2(\partial D)$ and as operators on $\big(L^2(\partial D)\big)^2$ (the latter defined by applying the operator coordinate-wise).
\begin{lemma}\label{lemma_t1t2} For every $\omega, \alpha$ and $\delta$,
\begin{equation}\label{eq:lemma_t1}
T_1^\alpha \mathcal{A}_\delta^{R^2\alpha,\omega} = \mathcal{A}_\delta^{\alpha,\omega}  T_1^\alpha, 
\end{equation}
and
\begin{equation} \label{eq:lemma_t2}
T_2 \mathcal{A}_\delta^{\alpha,\omega} = \mathcal{A}_\delta^{\alpha,\omega}  T_2.
\end{equation}

\end{lemma}
\pf 
To prove \eqnref{eq:lemma_t1}, recall that $R\Lambda^*=\Lambda^*$. Therefore, it follows that
$$
G^{R^2\alpha,k}(x-y) = \frac{1}{|Y|}\sum_{q\in \Lambda^*} \frac{e^{\iu(R^2\alpha+q)\cdot (x-y)}}{ k^2-|R^2\alpha+q|^2}=\frac{1}{|Y|}\sum_{q\in \Lambda^*} \frac{e^{\iu R(R^2\alpha+q)\cdot (Rx-Ry)}}{ k^2-|R\left(R^2\alpha+q\right)|^2} = G^{\alpha,k}(Rx-Ry).
$$
Then, we can check that
\begin{align*}
\Scal_{D_1}^{\alpha,k}[\psi(R_1y)](x)&= \int_{\p D_1} G^{\alpha,k}(x-y) \psi(R_1y) \dx\sigma(y)\\
&=\int_{\p D_1} G^{R^2\alpha,k}(R_1x-R_1y) \psi(R_1y) \dx\sigma(y) \\
&=\Scal_{D_1}^{R^2\alpha,k}[\psi(y)](R_1x)\\
&=\begin{cases} \Scal_{D_1}^{R^2\alpha,k}[\psi(y)](R_1x),\quad &x\in Y_1,\\
 e^{-\iu\alpha\cdot l_1}\Scal_{D_1}^{R^2\alpha,k}[\psi(y)](R_2x),\quad &x\in Y_2 .\end{cases}
\end{align*}
Similarly, we obtain that
\begin{align*}
\Scal_{D_2}^{\alpha,k}[\psi(R_2y)](x)&=\Scal_{D_2}^{R^2\alpha,k}[\psi(y)](R_2x)\\
&=\begin{cases} e^{\iu\alpha\cdot l_1}\Scal_{D_2}^{R^2\alpha,k}[\psi(y)](R_1x),\quad &x\in Y_1,\\
 \Scal_{D_2}^{R^2\alpha,k}[\psi(y)](R_2x),\quad &x\in Y_2 .\end{cases}
\end{align*}
Thus, we get
$$ \Scal_D^{\alpha,k}[T_1^\alpha\psi](x)=T_1^\alpha\Scal_D^{R^2\alpha,k}[\psi](x),\quad x\in Y.$$
This proves \eqnref{eq:lemma_t1}. To prove \eqnref{eq:lemma_t2}, we use the fact that
$$\overline{G^{\alpha,k}(x-y)} = G^{\alpha,k}\big( (2x_0-x) - (2x_0-y) \big).$$
Considering this together with the definitions of $\Scal_D^{\alpha,k}$ and $(\Kcal_D^{-\alpha,k})^*$, we find these operators commute with $T_2$, and hence $\Acal_{\delta}^{\alpha,\omega}$ commutes with $T_2$. This concludes the proof.
\qed

\begin{lemma} \label{ker_basis}
There are two elements $ \Phi_1$ and $\Phi_2$ in the kernel of  $\mathcal{A}_\delta^{\alpha^*,\omega^*}$ which satisfy
$$  T_1\Phi_1= \tau \Phi_1, \quad  T_1\Phi_2 =  \tau^2 \Phi_2, \quad  T_2 \Phi_1 = \Phi_2,$$
where $\tau = e^{2\pi \iu/3}$.
\end{lemma}

\pf Let $A$ be the kernel of $\mathcal{A}_\delta^{\alpha^*,\omega^*}$. By Lemma \ref{lemma_t1t2}, $ T_1$ and $T_2$ are operators from $A$ onto itself. Since $ T_1^3=I$ and $\ker ( T_1-I)$ is trivial in $A$ by the same argument as in the proof of Proposition \ref{prop_mul}, there is an element $\Phi\in A$ such that either $T_1\Phi = \tau \Phi$ or $ T_1\Phi = \tau^2 \Phi.$ In the first case we let $\Phi_1:=\Phi$ and $\Phi_2:= T_2 \Phi_1$. We then have 
$$T_1\Phi_2 = T_1( T_2\Phi_1)= T_2( T_1\Phi_1) = \tau^2 \Phi_2.$$ 
In the second case we let $\Phi_2 := \Phi$ and $\Phi_1 := T_2 \Phi$ and similarly find that $T_1\Phi_1 = \tau \Phi_1$. This proves the claim
\qed

Using the expansions \eqnref{eq:S1} and \eqnref{eq:K1}, we decompose $\Acal_\delta^{\alpha,\omega}$ into
\begin{align}
\Acal_\delta^{\alpha,\omega}&= \Acal_\delta^{\alpha^*,\omega}+ \begin{pmatrix} \mathcal{S}_1^{ k_b}& -\mathcal{S}_1^k \\ \mathcal{K}_1^{ k_b} & - \delta\mathcal{K}_1^k \end{pmatrix} \cdot (\alpha-\alpha^*) + O(|\alpha-\alpha^*|^2)\nonumber\\
&:= \Acal_\delta^{\alpha^*,\omega}+ \Acal_{\delta,1}^\omega \cdot (\alpha-\alpha^*)+ O(|\alpha-\alpha^*|^2),\label{A_exp}
\end{align}
uniformly for $\delta$ in a neighbourhood of $0$, with error term with respect to the operator norm. Here, $\cdot$ means the standard inner product taken component-wise. As shown in the proof of Lemma \ref{lem:first}, $\omega^*$ is a characteristic value of multiplicity 2, and from Proposition \ref{prop_mul} we know that $\ker\Acal_\delta^{\alpha^*,\omega^*}$ is two-dimensional. Consequently, $\omega^*$ is a pole of order 1 of $(\Acal_\delta^{\alpha^*,\omega})^{-1}$, and we can write
\begin{align} \label{eq:polepencil}
(\Acal_\delta^{\alpha^*,\omega})^{-1} = \frac{L}{\omega-\omega^*} 
+ E^\omega ,
\end{align}
where \hl{the operator $L$ is from $\big(L^2(\partial D)\big)^2$ onto $\ker\Acal_\delta^{\alpha^*,\omega^*}$} and $E^\omega$ is analytic in $\omega$.

Next, we investigate some properties of $L$. It is easy to check that $L$ vanishes on the range of $\Acal_\delta^{\alpha^*,\omega^*} $. By Lemma \ref{lemma_t1t2}, we have
\begin{equation*}
L  T_1 =  T_1 L,~~L T_2 =  T_2 L.
\end{equation*}
We also have the following result.
\begin{lemma}\label{rel_st}
For every $\alpha \in Y^*$ in a neighbourhood of $0$, it holds that
\begin{equation*}
L (\Acal_{\delta,1}^{\omega^*} \cdot \alpha) T_1\Phi =  T_1 L (\Acal_{\delta,1}^{\omega^*} \cdot  R^2\alpha)\Phi,
\end{equation*}
for every $\Phi$ in the kernel of $\Acal_\delta^{\alpha^*,\omega^*}$.
\end{lemma}
\pf By Lemma \ref{lemma_t1t2}, we have
\begin{equation}
\label{comm_1}
(\mathcal{A}_\delta^{\alpha,\omega} )^{-1}T_1^\alpha  =  T_1^\alpha (\mathcal{A}_\delta^{R^2\alpha,\omega})^{-1}.
\end{equation}
\hl{Moreover, since $R^2\alpha^* = \alpha^*$, we get from Lemma} \ref{lemma_t1t2} that
\begin{equation} \label{eq:comm_2}
\mathhl{T_1 \mathcal{A}_\delta^{\alpha^*,\omega} = \mathcal{A}_\delta^{\alpha^*,\omega}  T_1.}
\end{equation}
Using \eqnref{A_exp} and the Neumann series, we get, for fixed $\omega, \delta$,
\begin{align}
(\mathcal{A}_\delta^{\alpha,\omega} )^{-1}&= (\mathcal{A}_\delta^{\alpha^*,\omega} )^{-1}
+(\mathcal{A}_\delta^{\alpha^*,\omega} )^{-1} \Acal_{\delta,1}^\omega \cdot (\alpha-\alpha^*) (\mathcal{A}_\delta^{\alpha^*,\omega} )^{-1} + O(|\alpha-\alpha^*|^2), \label{eq:Aexp}\\
(\mathcal{A}_\delta^{R^2\alpha,\omega} )^{-1}&= (\mathcal{A}_\delta^{\alpha^*,\omega} )^{-1}
+(\mathcal{A}_\delta^{\alpha^*,\omega} )^{-1} \Acal_{\delta,1}^\omega \cdot (R^2\alpha-R^2\alpha^*) (\mathcal{A}_\delta^{\alpha^*,\omega} )^{-1} + O(|\alpha-\alpha^*|^2), \label{eq:AR2exp}
\end{align}
where error terms are in the operator norm. \hl{We also expand $T_1^\alpha$ as}
\begin{equation}\label{eq:Texp}
\mathhl{T_1^\alpha=T_1 + \hat T_1\cdot(\alpha-\alpha^*)+ O(|\alpha-\alpha^*|^2),}\end{equation}
where the error is in the operator norm and $\hat T_1$ is given by
$$
(\hat T_1 f)(x):= \begin{cases} -\iu \tau l_1 f(R_1x), &x\in Y_1,\\
-2\iu \tau^2 l_1 f(R_2x), &x\in Y_2, \end{cases}
$$
\hl{Substituting} the asymptotic expansions \eqnref{eq:Aexp}, \eqnref{eq:AR2exp} and \eqnref{eq:Texp} into \eqnref{comm_1}, and collecting terms of order $O(|\alpha-\alpha^*|)$, we get
\begin{align*}
&(\mathcal{A}_\delta^{\alpha^*,\omega} )^{-1} \hat T_1\cdot(\alpha-\alpha^*) + (\mathcal{A}_\delta^{\alpha^*,\omega} )^{-1}  \Acal_{\delta,1}^\omega \cdot (\alpha-\alpha^*) (\mathcal{A}_\delta^{\alpha^*,\omega} )^{-1} T_1 \\
&\qquad = \hat T_1\cdot(\alpha-\alpha^*)(\mathcal{A}_\delta^{\alpha^*,\omega} )^{-1}  + T_1(\mathcal{A}_\delta^{\alpha^*,\omega} )^{-1} \Acal_{\delta,1}^\omega \cdot (R^2\alpha-R^2\alpha^*)  (\mathcal{A}_\delta^{\alpha^*,\omega} )^{-1}.
\end{align*}
Applying $\mathcal{A}_\delta^{\alpha^*,\omega}$ on the above identity \hl{from the right, and using} \eqnref{eq:comm_2}, we obtain
\begin{align*}
&(\mathcal{A}_\delta^{\alpha^*,\omega} )^{-1} \hat T_1\cdot(\alpha-\alpha^*) \mathcal{A}_\delta^{\alpha^*,\omega}+ (\mathcal{A}_\delta^{\alpha^*,\omega} )^{-1}  \Acal_{\delta,1}^\omega \cdot (\alpha-\alpha^*)  T_1  \\
&\qquad = \hat T_1\cdot(\alpha-\alpha^*) + T_1(\mathcal{A}_\delta^{\alpha^*,\omega} )^{-1}  \Acal_{\delta,1}^\omega \cdot (R^2\alpha-R^2\alpha^*).
\end{align*}
Integrating over a small contour around $\omega^*$, we have
$$
L \hat T_1\cdot(\alpha-\alpha^*) \mathcal{A}_\delta^{\alpha^*,\omega^*}+ L \Acal_{\delta,1}^{\omega^*} \cdot (\alpha-\alpha^*)  T_1 
=  T_1L  \Acal_{\delta,1}^{\omega^*} \cdot (R^2\alpha-R^2\alpha^*) .
$$
On the kernel of $\Acal_\delta^{\alpha^*,\omega^*}$, the first term vanishes. Therefore, we get
$$
L \Acal_{\delta,1}^{\omega^*} \cdot (\alpha-\alpha^*)  T_1 
=  T_1L  \Acal_{\delta,1}^{\omega^*} \cdot (R^2\alpha-R^2\alpha^*),
$$
on the kernel of $\Acal_\delta^{\alpha^*,\omega^*}$. This completes the proof. 
\qed

\begin{lemma} \label{lem:err}
	For $\omega = O(\sqrt{\delta})$ with $\omega-\omega^* = \mu \sqrt{\delta}$ for a fixed $\mu \neq 0$, we have
	$$\left\|(\Acal_\delta^{\alpha^*,\omega})^{-1} (\Acal_\delta^{\alpha^*,\omega}-\Acal_\delta^{\alpha,\omega})\right\|_{\B\big((L^2(\p D))^2,(L^2(\p D))^2\big)} = O(|\alpha-\alpha^*|),$$
	uniformly for $\delta$ in a neighbourhood of $0$.
\end{lemma}
\pf We first observe, from \eqnref{eq:sing}, that
\begin{equation} \label{eq:Asing}
\left\|(\Acal_\delta^{\alpha^*,\omega})^{-1}\right\| = O(\delta^{-1}),
\end{equation}
where we use the shorthand $\| \cdot \|$ for the norm in $\B\big(\big(L^2(\p D)\big)^2,\big(L^2(\p D)\big)^2\big)$. We know that $L$ vanishes on the range of $\Acal_\delta^{\alpha^*,\omega^*}$ and similarly,  the adjoint $L^*$ vanishes on the range of $\left(\Acal_\delta^{\alpha^*,\omega^*} \right)^*$. Therefore, $L$ can be written as 
\begin{equation} \label{eq:T}
L = \langle \mathcal{X}_1, \cdot \rangle \Phi_1 + \langle \mathcal{X}_2, \cdot \rangle \Phi_2\end{equation}
for some $\mathcal{X}_1, \mathcal{X}_2 \in \ker\left(\Acal_\delta^{\alpha^*,\omega^*} \right)^*$. Here, $\langle \cdot, \cdot \rangle$ denotes the standard inner product in $\big(L^2(\p D)\big)^2$. Using Lemma \ref{lem:first} (i), it is straightforward to check that a basis for $\ker\left(\Acal_\delta^{\alpha^*,\omega^*} \right)^*$ is given by
$$\left\{\begin{pmatrix} 0 \\ \chi_{\p D_1} \end{pmatrix} + O(\delta), \ \begin{pmatrix} 0 \\ \chi_{\p D_2} \end{pmatrix}+ O(\delta)\right\},$$
where the error is with respect to the norm in $\big(L^2(\p D)\big)^2$. Therefore, from \eqnref{eq:Asing} and \eqnref{eq:polepencil} we have
\begin{equation}\label{eq:Chi}
\mathcal{X}_i = \frac{1}{\sqrt{\delta}}\begin{pmatrix} 0 \\ \chi_i \end{pmatrix} + O(\sqrt{\delta}), \ i = 1,2,
\end{equation}
where $\chi_i(x)$ is a function that is constant for $x \in \p D_j, \ j=1,2$ and satisfies $\|\chi_i\|_{L^2(\p D)} = O(1)$. Combining Lemma \ref{lem:layer} (ii) with expansion \eqnref{eq:SKexp}, we find using \eqnref{eq:T} and \eqnref{eq:Chi} that 
\begin{equation}\label{eq:TA}
\|L\Acal_{\delta}^{\alpha,\omega}\| = O(\sqrt{\delta}),
\end{equation}
uniformly for $\alpha \in Y^*_0$. Since $\left\|\Acal_\delta^{\alpha^*,\omega}-\Acal_\delta^{\alpha,\omega}\right\| = O(|\alpha-\alpha^*|) $ uniformly for $\delta$ in a neighbourhood of $0$, we obtain that 
\begin{align*}
\left\|(\Acal_\delta^{\alpha^*,\omega})^{-1} (\Acal_\delta^{\alpha^*,\omega}-\Acal_\delta^{\alpha,\omega})\right\| &= \frac{1}{\mu \sqrt{\delta}} \left\|L (\Acal_\delta^{\alpha^*,\omega}-\Acal_\delta^{\alpha,\omega})\right\| + O(|\alpha-\alpha^*|) \\
&= O(|\alpha-\alpha^*|),
\end{align*}
uniformly for $\delta$ in a neighbourhood of $0$. This concludes the proof.
\qed

\subsection{Proof of Theorem \ref{main}} \label{sec:pf}
Now, we are ready to prove our main result in this paper, namely Theorem \ref{main}. We first observe that we only need to prove \eqnref{eq:cone1} and \eqnref{eq:cone2}, the remaining statements then follow from \eqnref{eq:asymptotic} and Lemma \ref{lem:c}. Throughout the proof, we assume that $\delta$ is sufficiently small so that Proposition \ref{prop_mul} holds.

Let $V\subset \mathbb{C}$ be a neighbourhood of $\omega^*$ containing only the two characteristic values $\omega_j^\alpha(\delta)$ for $j=1,2$, and satisfying $|\p V| = O(\sqrt{\delta})$. Then the generalized argument principle \cite[Theorem 1.14]{surv235} tells us that the characteristic values $\omega_1^\alpha(\delta)$ and $\omega_2^\alpha(\delta)$  of $\Acal_\delta^{\alpha,\omega}$ near $\alpha^*$ satisfy
$$ f(\omega_1^\alpha(\delta))+f(\omega_2^\alpha(\delta))=\frac{\mbox{tr}}{2\pi \iu}\int_{\p V}  (\Acal_\delta^{\alpha,\omega})^{-1} \frac{\mathrm{d}}{\dx\omega}\Acal_\delta^{\alpha,\omega} f(\omega)\dx\omega, $$
for any analytic function $f(\omega)$. As in the proof of \cite[Theorem 3.9]{akl}, we have
\begin{align*}
f(\omega_1^\alpha(\delta))+f(\omega_2^\alpha(\delta))=-\frac{\mbox{tr}}{2\pi \iu}\sum_{p=1}^\infty \int_{\p V} \frac{f(\omega)}{p} \frac{\mathrm{d}}{\dx\omega}\left[ (\Acal_\delta^{\alpha^*,\omega})^{-1} (\Acal_\delta^{\alpha^*,\omega}-\Acal_\delta^{\alpha,\omega})   \right]^p \dx\omega.
\end{align*}
\hl{Integrating by parts, we find}
\begin{align}\label{eq:gap}
\mathhl{f(\omega_1^\alpha(\delta))+f(\omega_2^\alpha(\delta))=\frac{\mbox{tr}}{2\pi \iu}\sum_{p=1}^\infty \int_{\p V} \frac{\mathrm{d} f(\omega)}{\dx\omega} \frac{1}{p} \left[ (\Acal_\delta^{\alpha^*,\omega})^{-1} (\Acal_\delta^{\alpha^*,\omega}-\Acal_\delta^{\alpha,\omega})   \right]^p \dx\omega.}
\end{align}
Using \eqnref{eq:gap} twice, with $f(\omega) = \omega-\omega^*$ and $f(\omega) =\left(\omega-\omega^*\right)^2$, and using Lemma \ref{lem:err} we have
\begin{align}
\omega_1^\alpha(\delta)- \omega^* + \omega_2^\alpha(\delta)-\omega^*&= \frac{\mbox{tr}}{2\pi \iu}\int_{\p V}  (\Acal_\delta^{\alpha^*,\omega})^{-1} \Acal_{\delta,1}^{\omega}\cdot (\alpha-\alpha^*)\dx\omega + O(\sqrt{\delta}|\alpha-\alpha^*|^2), \label{eq:wlin} \\
(\omega_1^\alpha(\delta)- \omega^*)^2 + (\omega_2^\alpha(\delta)-\omega^*)^2&= \frac{\mbox{tr}}{2\pi \iu}\int_{\p V} (\omega-\omega^*)\left[ (\Acal_\delta^{\alpha^*,\omega})^{-1} \Acal_{\delta,1}^{\omega}\cdot (\alpha-\alpha^*)\right]^2 \dx\omega \nonumber \\
&\qquad + O(\delta|\alpha-\alpha^*|^3), \label{eq:wsquare}
\end{align}
where we have used Cauchy's theorem to conclude that the term corresponding to $p=1$ vanishes in \eqnref{eq:wsquare}. Here, the error terms hold uniformly for $\delta$ in a neighbourhood of $0$ and $\alpha \in Y^*_0$. To finish the proof, it suffices to show that
\begin{align}
\frac{\mbox{tr}}{2\pi \iu}\int_{\p V}  (\Acal_\delta^{\alpha^*,\omega})^{-1} \Acal_{\delta,1}^{\omega}\cdot (\alpha-\alpha^*)\dx\omega&=0, \label{eq:firsttoprove}\\
 \frac{\mbox{tr}}{2\pi \iu}\int_{\p V} (\omega-\omega^*)\left[ (\Acal_\delta^{\alpha^*,\omega})^{-1} \Acal_{\delta,1}^{\omega}\cdot (\alpha-\alpha^*)\right]^2 \dx\omega&= C|\alpha-\alpha^*|^2, \label{eq:secondtoprove}
\end{align}
for some $C$ which is constant in $\alpha$ and scales as $O(\delta)$ as $\delta \rightarrow 0$. Indeed, this together with \eqnref{eq:wlin} and \eqnref{eq:wsquare} would imply that
$$\omega_j^\alpha(\delta) = \omega^* \pm \lambda|\alpha-\alpha^*|[1+O(|\alpha-\alpha^*|)],$$
uniformly for $\delta$ in a neighbourhood of $0$, where $2\lambda^2 = C$. The expression for $\lambda$ then follows from \eqnref{eq:asymptotic}. We see that 
\begin{align}\label{eq:first}
\frac{1}{2\pi \iu}\int_{\p V}  (\Acal_\delta^{\alpha^*,\omega})^{-1} \Acal_{\delta,1}^{\omega}\cdot (\alpha-\alpha^*)\dx\omega = L\Acal_{\delta,1}^{\omega^*}\cdot(\alpha-\alpha^*), 
\end{align}
is an operator \hl{that maps the kernel of $\Acal_\delta^{\alpha^*,\omega^*}$ onto itself}. Similarly, we get
\begin{align}\label{eq:second}
\frac{1}{2\pi \iu}\int_{\p V} (\omega-\omega^*)\left[ (\Acal_\delta^{\alpha^*,\omega})^{-1} \Acal_{\delta,1}^{\omega}\cdot (\alpha-\alpha^*)\right]^2 \dx\omega =[L\Acal_{\delta,1}^{\omega^*}\cdot(\alpha-\alpha^*)]^2.
\end{align}
For $\widetilde\alpha = \left( \begin{smallmatrix} \widetilde\alpha_{(1)} \\ \widetilde\alpha_{(2)} \end{smallmatrix} \right)$ in a neighbourhood of $0$, let
\begin{equation*}
\Acal_{\delta,1}^{\omega^*}\cdot \widetilde\alpha= \Acal_{\delta,11}^{\omega^*}\widetilde\alpha_{(1)} + \Acal_{\delta,12}^{\omega^*} \widetilde\alpha_{(2)}.
\end{equation*}
Suppose that
\begin{align*}
L\Acal_{\delta,11}^\omega[\Phi_1] = a \Phi_1 + b \Phi_2,\\
L\Acal_{\delta,12}^\omega[\Phi_1] = c \Phi_1 + d \Phi_2.
\end{align*}
From \eqnref{eq:TA} we know that $a,b,c$ and $d$ scale as $O(\sqrt{\delta})$ as $\delta \rightarrow 0$. Since
$$R^2= \frac{1}{2} \begin{pmatrix}-1 &  - \sqrt{3} \\ \sqrt{3} & -1 \end{pmatrix} $$
and $ T_1 \Phi_1 = \tau \Phi_1$, we obtain
\begin{align*}
L(\Acal_{\delta,1}^{\omega^*} \cdot\widetilde \alpha) T_1[\Phi_1]&=\tau \big(\widetilde \alpha_{(1)}(a \Phi_1 + b \Phi_2)+\widetilde\alpha_{(2)}(c \Phi_1 + d \Phi_2) \big) \\
 T_1 L(\Acal_{\delta,1}^{\omega^*} \cdot  R^2\widetilde\alpha)[\Phi_1]&=  T_1 \left[ \frac{-\widetilde\alpha_{(1)} - \sqrt{3}\widetilde\alpha_{(2)}}{2}(a \Phi_1 + b \Phi_2)+\frac{\sqrt{3}\widetilde\alpha_{(1)}-\widetilde\alpha_{(2)}}{2} (c \Phi_1 + d \Phi_2) \right] \\
&=\frac{-\widetilde\alpha_{(1)} - \sqrt{3}\widetilde\alpha_{(2)}}{2}(a \tau\Phi_1 + b\tau^2 \Phi_2)+\frac{\sqrt{3}\widetilde\alpha_{(1)}-\widetilde\alpha_{(2)}}{2} (c \tau\Phi_1 + d \tau^2\Phi_2).
\end{align*}
By Lemma \ref{rel_st},  it follows that
\begin{align*}
2a= -a+\sqrt{3}c,~ 2b=-\tau b+\sqrt{3}\tau d,~2c=-\sqrt{3}a -c,~2d=-\sqrt{3}\tau b - \tau d.
\end{align*}
Solving these equations, we obtain $a=c=0,~ d= -\frac{\sqrt{3} \tau}{2+\tau} b=-\iu b$, which means
\begin{equation*}
L\Acal_{\delta,11}^{\omega^*}[\Phi_1] = b \Phi_2,\quad L\Acal_{\delta,12}^{\omega^*}[\Phi_1] = - \iu b \Phi_2.
\end{equation*}
Similarly, putting 
\begin{align*}
L\Acal_{\delta,11}^\omega[\Phi_2] = \tilde a \Phi_1 + \tilde b \Phi_2,\\
L\Acal_{\delta,12}^\omega[\Phi_2] = \tilde c \Phi_1 + \tilde d \Phi_2.
\end{align*}
we carry out the same steps to conclude that
\begin{equation*}
L\Acal_{\delta,11}^{\omega^*}[\Phi_2] = \tilde  a \Phi_1, \quad L\Acal_{\delta,12}^{\omega^*}[\Phi_2] =  \iu\tilde a \Phi_1,
\end{equation*}
where $\tilde a = O(\sqrt{\delta})$. Now, we arrive at 
\begin{align*}
L \Acal_{\delta,1}^{\omega^*} \cdot (\alpha-\alpha^*) [\Phi_1] =  b\left((\alpha_{(1)}-\alpha_{(1)}^*)-\iu(\alpha_{(2)} -\alpha_{(2)}^*) \right)\Phi_2,\\
L \Acal_{\delta,1}^{\omega^*} \cdot (\alpha-\alpha^*) [\Phi_2] =  \tilde a\left((\alpha_{(1)}-\alpha_{(1)}^*)+\iu(\alpha_{(2)} -\alpha_{(2)}^*) \right)\Phi_1.
\end{align*}
Therefore, we can conclude that
\begin{align*} \mbox{tr}\,L \Acal_{\delta,1}^{\omega^*} \cdot (\alpha-\alpha^*)&=0,\\
 \mbox{tr}\,[L\Acal_{\delta,1}^{\omega^*} \cdot (\alpha-\alpha^*)]^2 &= 2\tilde ab |\alpha-\alpha^*|^2.
\end{align*}
This, together with  \eqnref{eq:first} and \eqnref{eq:second}, proves equations \eqnref{eq:firsttoprove} and \eqnref{eq:secondtoprove}, which completes the proof. \qed 

\section{Numerical illustrations} \label{sec:num} 
We consider the bubbly honeycomb crystal as described in Section \ref{sec:setup}, and additionally assume that the bubbles are circular with radius $R$. The center-to-center distance between adjacent bubbles is assumed to be one and  the material parameters are such that $v = v_b = 1$. The lattice basis vectors are given by 
$$ l_1 = \left( 3, \sqrt{3} \right),~~l_2 = \left( 3, -\sqrt{3} \right),$$
\begin{figure}
	\centering
	\begin{tikzpicture}[scale=1.3]	
	\coordinate (a) at ({1/sqrt(3)},1);		
	\coordinate (b) at ({1/sqrt(3)},-1);
	\coordinate (c) at ({2/sqrt(3)},0);
	\coordinate (M) at ({0.5/sqrt(3)},0.5);
	\coordinate (K1) at ({1/sqrt(3)},{1/3});
	\coordinate (K2) at ({1/sqrt(3)},{-1/3});
	\coordinate (K3) at (0,{-2/3});
	\coordinate (K4) at ({-1/sqrt(3)},{-1/3});
	\coordinate (K5) at ({-1/sqrt(3)},{1/3});
	\coordinate (K6) at (0,{2/3});
	
	\draw[->] (0,0) -- (a) node[above]{$\alpha_1$};
	\draw[->] (0,0) -- (b) node[below]{$\alpha_2$};
	\draw (a) -- (c) -- (b) node[pos=0.4,below right]{$Y^*$};
	\draw[fill] (M) circle(1pt) node[yshift=8pt, xshift=-2pt]{$M$}; 
	\draw[fill] (0,0) circle(1pt) node[left]{$\Gamma$}; 
	\draw[fill] (K1) circle(1pt) node[right]{$K$}; 
	
	\draw[postaction={decorate}, decoration={markings, mark=at position 0.2 with {\arrow{>}}, markings, mark=at position 0.65 with {\arrow{>}}, markings, mark=at position 0.9 with {\arrow{>}}}, color=red]
	(M) -- (0,0) -- (K1) -- (M);
	\draw[opacity=0.4] (K1) -- (K2) -- (K3) -- (K4) -- (K5) -- (K6) -- cycle; 
	\end{tikzpicture}
	\caption{Symmetry points in the reciprocal space, and the path along which the band structure is numerically computed.} \label{fig:BZband}
\end{figure}
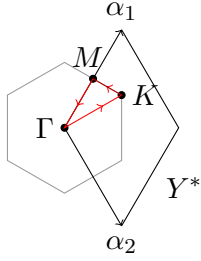
and the reciprocal basis vectors are defined as
$$ \alpha_1 = 2\pi\left( \frac{1}{6}, \frac{1}{2\sqrt{3}} \right),~~\alpha_2 = 2\pi\left(\frac{1}{6}, -\frac{1}{2\sqrt{3}} \right).$$
Using the same notation as in Section \ref{sec:setup}, this corresponds to $a = 2\sqrt{3}$. We also define the symmetry points in the reciprocal space as follows:
$$
\Gamma  = (0,0), \quad K = \frac{2\alpha_1 + \alpha_2}{3} ,  \quad M = \frac{\alpha_1}{2}.
$$
	
We use the multipole expansion method to compute the band diagrams \cite{AFLYZ}. Because $D$ has two connected components $D_1$ and $D_2$, we can identify $L^2(\partial D) = L^2(\partial D_1)\times L^2(\partial D_2)$. This gives the following matrix expressions of $\Scal_D^{\alpha,k}$ and $(\Kcal_D^{-\alpha,k})^*$:
\begin{align*}
\Scal_D^{\alpha,k}[\phi] = \begin{pmatrix}
\Scal_{D_1}^{\alpha,k} & \Scal_{D_2}^{\alpha,k} \\
\Scal_{D_1}^{\alpha,k} & \Scal_{D_2}^{\alpha,k}
\end{pmatrix} \begin{pmatrix}
\phi_{(1)} \\ \phi_{(2)}
\end{pmatrix},\quad 
(\Kcal_D^{-\alpha,k})^*[\phi] = \begin{pmatrix}
(\Kcal_{D_1}^{-\alpha,k})^* & \frac{\p}{\p \nu}\Scal_{D_2}^{\alpha,k}  \\
\frac{\p}{\p \nu}\Scal_{D_1}^{\alpha,k} & (\Kcal_{D_2}^{-\alpha,k})^*
\end{pmatrix} \begin{pmatrix}
\phi_{(1)} \\ \phi_{(2)}
\end{pmatrix}.
\end{align*}
Here, $\phi \in L^2(\p D)$ is represented by $\phi = \left( \begin{smallmatrix} \phi_{(1)} \\ \phi_{(2)} \end{smallmatrix} \right) \in L^2(\p D_1) \times L^2(\p D_2)$. Using these expressions, the integral operator $\Acal_\delta^{\alpha,\omega}$ defined in equation \eqnref{eq:A} can be discretised with the multipole expansion method as described in \cite[Appendix C]{AFLYZ}. We consider the band structure along the line $M\Gamma KM$, illustrated in Figure \ref{fig:BZband}, in the following numerical examples:
\begin{itemize}

\item[(i)] 
(Dilute regime).
We set $R=1/50$ and $\delta=1/9000$. The band structure is given in Figure \ref{fig:honeycomb_R_1_50}. The left subfigure shows the first four bands. The right subfigure shows the first two bands, which correspond to subwavelength curves and which cross at $K$. Observe that the crossing is a linear dispersion which means that it signifies a Dirac point.

\item[(ii)] 
(Non-dilute regime)
We set $R=1/5$ and $\delta=1/1000$. The band structure is given in Figure \ref{fig:honeycomb_R_1_5}. In this non-dilute regime, there is still a Dirac cone at the point $K$.

\end{itemize}

\begin{figure} [!h]
\begin{center}
 \includegraphics[height=5.0cm]{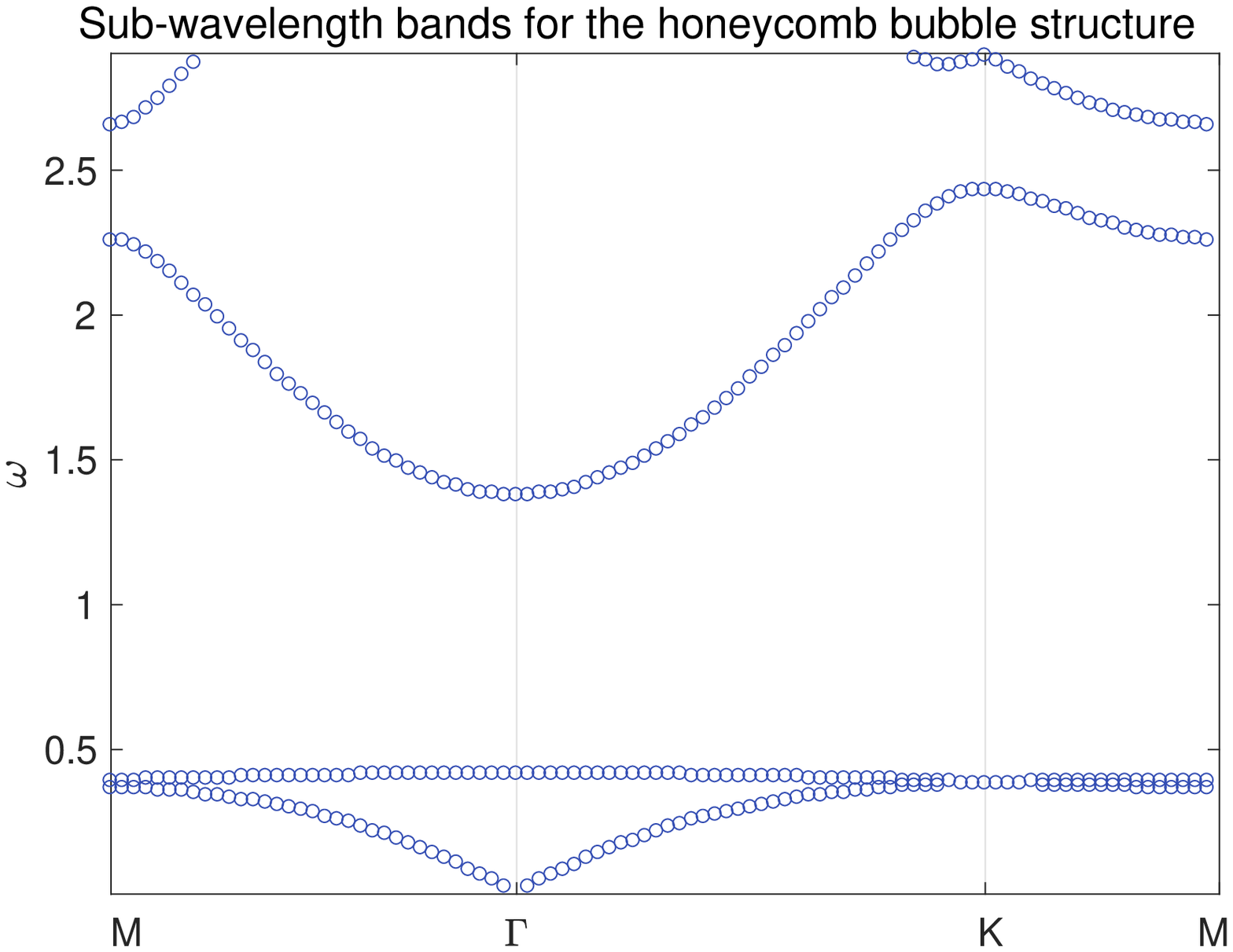}
 \hskip-.5cm
  \includegraphics[height=5.0cm]{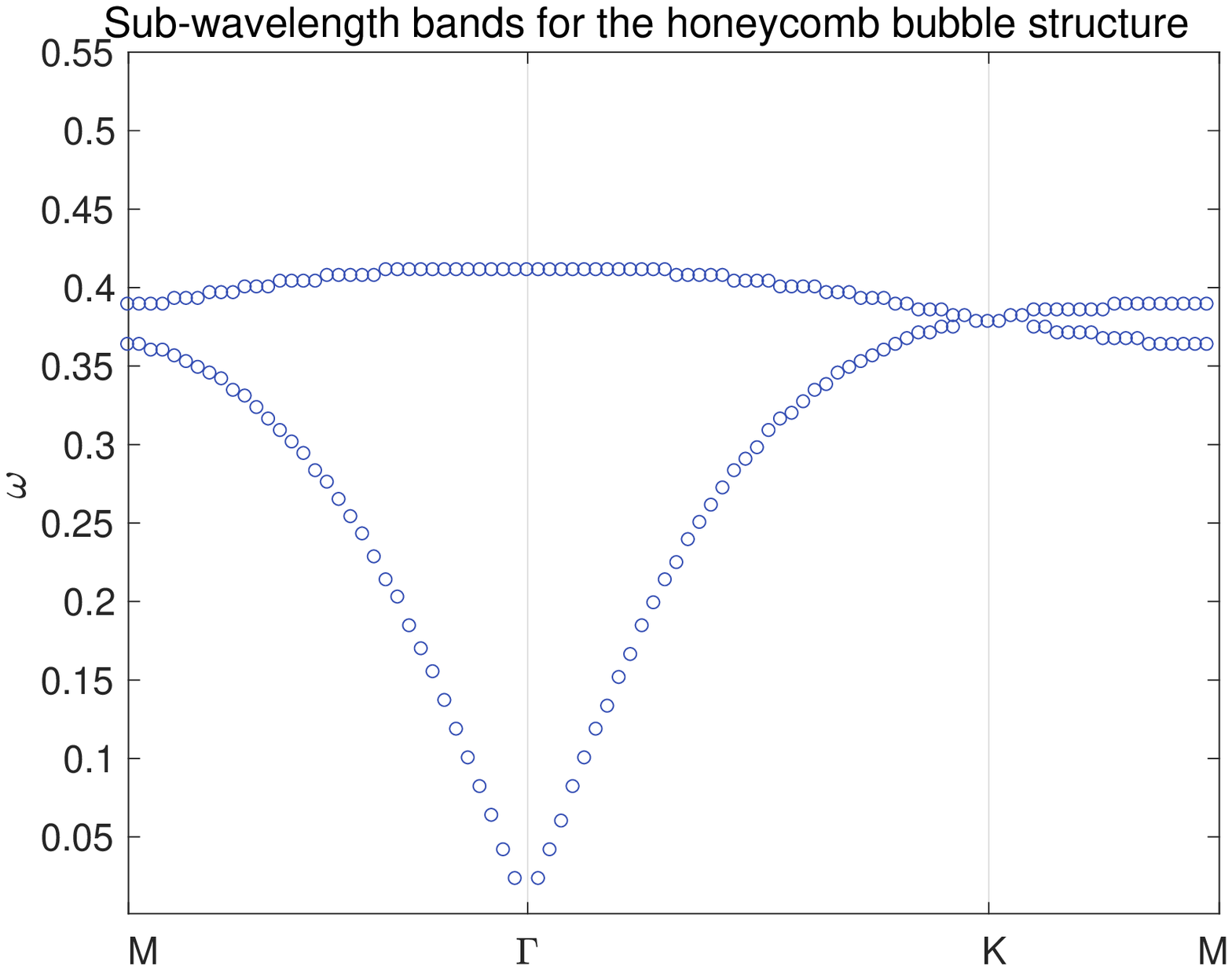}
   \caption{ 
  (left) The band structure of a bubbly honeycomb phononic crystal  with $R=1/50$ and $\delta = 1/9000$. The distance between the adjacent bubbles is one; (right) The band structure upon zooming in on the subwavelength region.}
  \label{fig:honeycomb_R_1_50}
\end{center}
\end{figure}

\begin{figure} [!h]
\begin{center}
 \includegraphics[height=5.0cm]{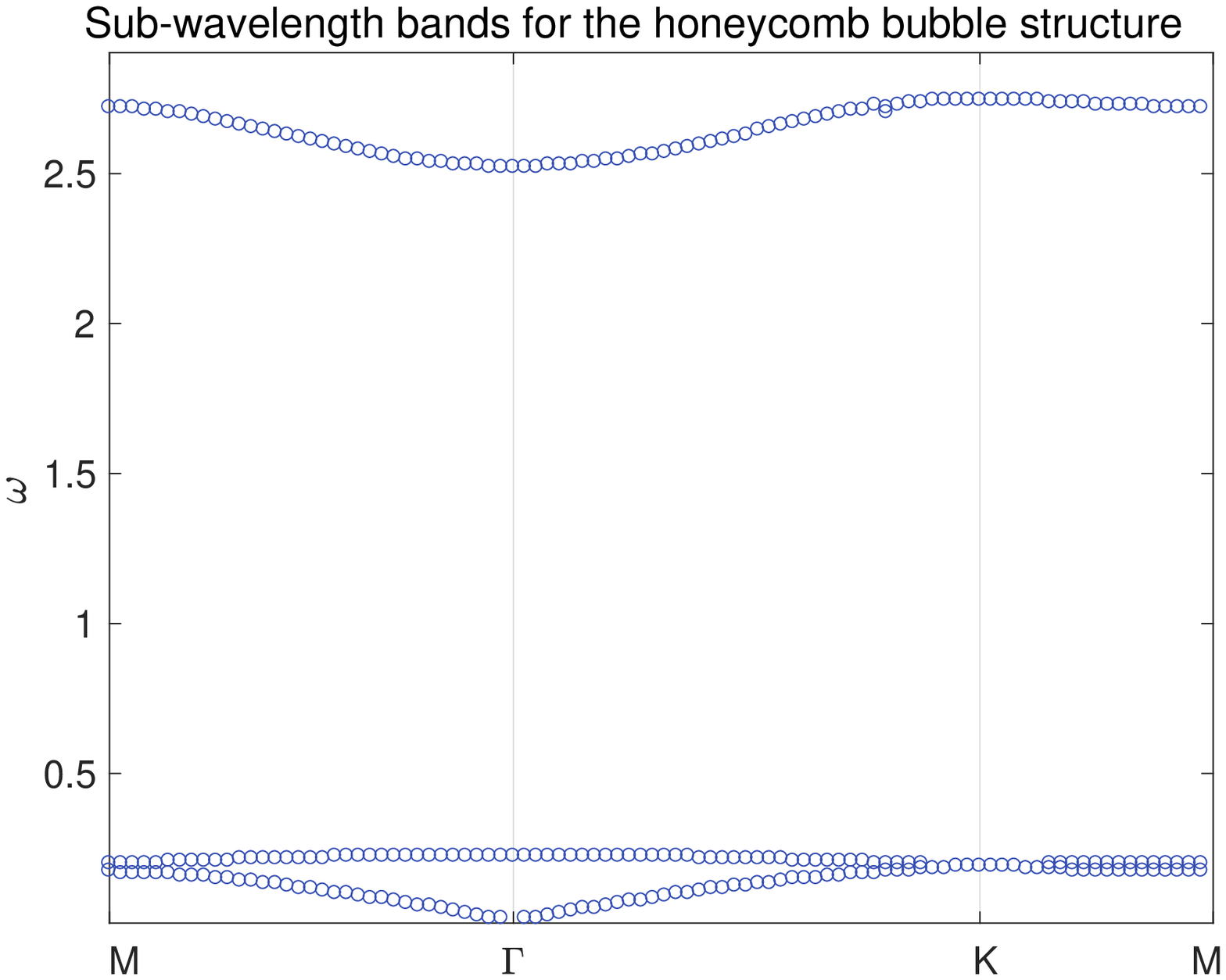}
 \hskip-.5cm
  \includegraphics[height=5.0cm]{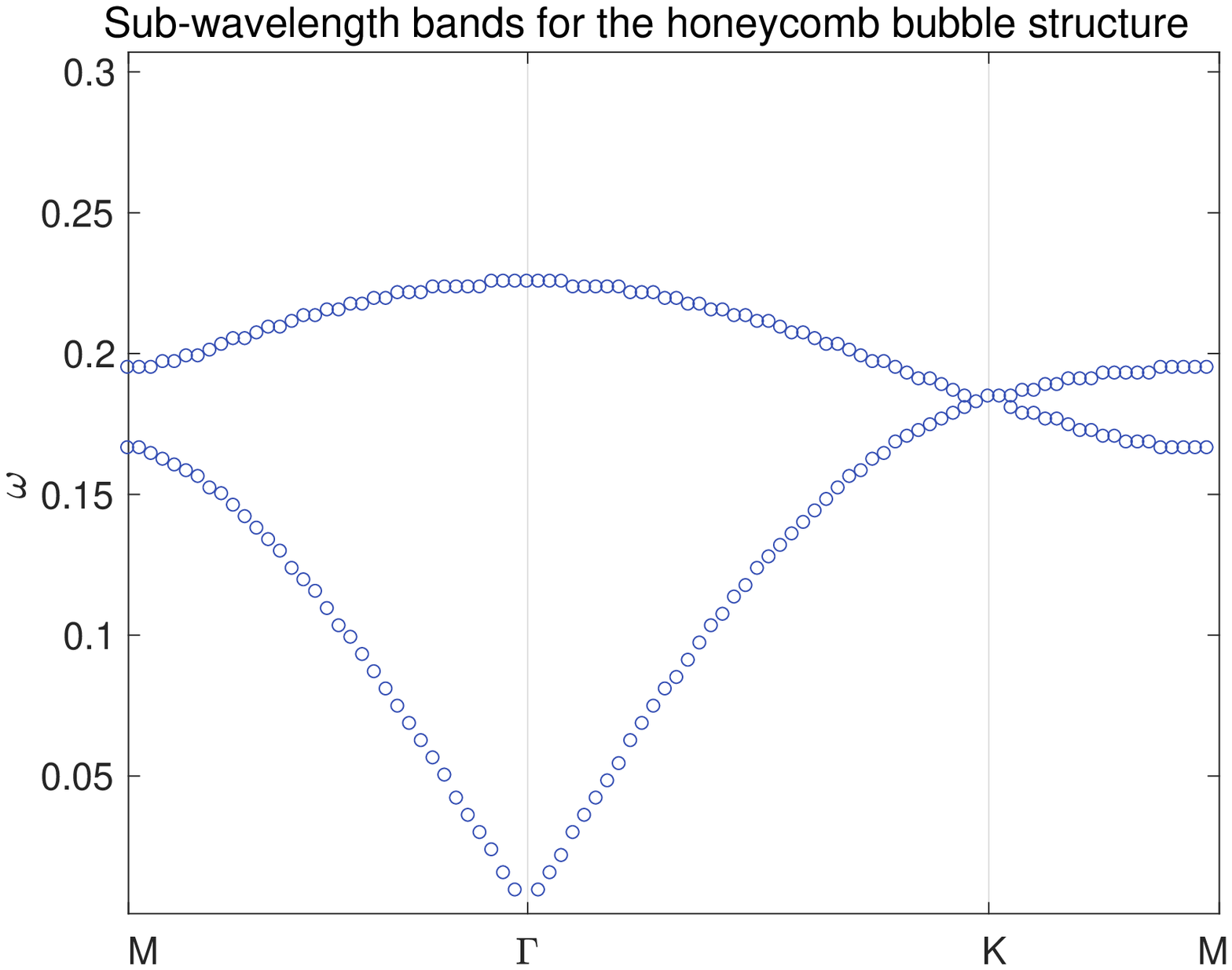}
   \caption{ 
  (left)   The band structure of a bubbly honeycomb phononic crystal  with  $R=1/5$ and $\delta = 1/1000$. The distance between the adjacent bubbles is one; (right) The band structure upon zooming in on the subwavelength region.}
  \label{fig:honeycomb_R_1_5}
\end{center}
\end{figure}

\section{Concluding remarks} \label{sec:conclusion}
In this paper, we have rigorously proven the existence of a Dirac cone in the subwavelength regime in a bubbly phononic crystal with a honeycomb lattice structure. We have illustrated our main results with different numerical experiments. In view of the recent results in \cite{pierre,matias1,matias2}, our original approach in this work can be extended to plasmonics. 
In future works, we plan to further study topological phenomena in bubbly crystals. In particular, we will rigorously show the existence of localized edge states at the surface of a topologically non-trivial bubbly crystal. Similar to \cite{homogenization}, a high-frequency homogenization of a bubbly honeycomb phononic crystal is performed in \cite{arma_bubble}.

\end{document}